\documentclass[11pt]{amsart}
\usepackage{amssymb}
\usepackage{amsmath}
\usepackage{txfonts}

\usepackage[all,cmtip]{xy}

\usepackage{amsfonts}
\usepackage{mathrsfs}
\usepackage{latexsym}
\usepackage{graphicx}
\usepackage{amscd,amssymb,amsmath,amsbsy,amsthm}

\usepackage[colorlinks,plainpages,backref,urlcolor=blue]{hyperref}
\usepackage{verbatim}

\usepackage{tikz}
\usetikzlibrary{arrows,calc}
\tikzset{
>=stealth',
help lines/.style={dashed, thick},
axis/.style={<->},
important line/.style={thick},
connection/.style={thick, dotted},
}

\newcommand{\nc}{\newcommand}
\nc{\rnc}{\renewcommand}
\nc{\bb}[1]{\mathbb{#1}}
\nc{\bbA}{\bb{A}}\nc{\bbB}{\bb{B}}\nc{\bbC}{\bb{C}}\nc{\bbD}{\bb{D}}
\nc{\bbE}{\bb{E}}\nc{\bbF}{\bb{F}}\nc{\bbG}{\bb{G}}\nc{\bbH}{\bb{H}}
\nc{\bbI}{\bb{I}}\nc{\bbJ}{\bb{J}}\nc{\bbK}{\bb{K}}\nc{\bbL}{\bb{L}}
\nc{\bbM}{\bb{M}}\nc{\bbN}{\bb{N}}\nc{\bbO}{\bb{O}}\nc{\bbP}{\bb{P}}
\nc{\bbQ}{\bb{Q}}\nc{\bbR}{\bb{R}}\nc{\bbS}{\bb{S}}\nc{\bbT}{\bb{T}}
\nc{\bbU}{\bb{U}}\nc{\bbV}{\bb{V}}\nc{\bbW}{\bb{W}}\nc{\bbX}{\bb{X}}
\nc{\bbY}{\bb{Y}}\nc{\bbZ}{\bb{Z}}
\nc{\mbf}[1]{{\mathbf #1}}
\nc{\bfA}{\mbf{A}}\nc{\bfB}{\mbf{B}}\nc{\bfC}{\mbf{C}}\nc{\bfD}{\mbf{D}}
\nc{\bfE}{\mbf{E}}\nc{\bfF}{\mbf{F}}\nc{\bfG}{\mbf{G}}\nc{\bfH}{\mbf{H}}
\nc{\bfI}{\mbf{I}}\nc{\bfJ}{\mbf{J}}\nc{\bfK}{\mbf{K}}\nc{\bfL}{\mbf{L}}
\nc{\bfM}{\mbf{M}}\nc{\bfN}{\mbf{N}}\nc{\bfO}{\mbf{O}}\nc{\bfP}{\mbf{P}}
\nc{\bfQ}{\mbf{Q}}\nc{\bfR}{\mbf{R}}\nc{\bfS}{\mbf{S}}\nc{\bfT}{\mbf{T}}
\nc{\bfU}{\mbf{U}}\nc{\bfV}{\mbf{V}}\nc{\bfW}{\mbf{W}}\nc{\bfX}{\mbf{X}}
\nc{\bfY}{\mbf{Y}}\nc{\bfZ}{\mbf{Z}}
\nc{\bfa}{\mbf{a}}\nc{\bfb}{\mbf{b}}\nc{\bfc}{\mbf{c}}\nc{\bfd}{\mbf{d}}
\nc{\bfe}{\mbf{e}}\nc{\bff}{\mbf{f}}\nc{\bfg}{\mbf{g}}\nc{\bfh}{\mbf{h}}
\nc{\bfi}{\mbf{i}}\nc{\bfj}{\mbf{j}}\nc{\bfk}{\mbf{k}}\nc{\bfl}{\mbf{l}}
\nc{\bfm}{\mbf{m}}\nc{\bfn}{\mbf{n}}\nc{\bfo}{\mbf{o}}\nc{\bfp}{\mbf{p}}
\nc{\bfq}{\mbf{q}}\nc{\bfr}{\mbf{r}}\nc{\bfs}{\mbf{s}}\nc{\bft}{\mbf{t}}
\nc{\bfu}{\mbf{u}}\nc{\bfv}{\mbf{v}}\nc{\bfw}{\mbf{w}}\nc{\bfx}{\mbf{x}}
\nc{\bfy}{\mbf{y}}\nc{\bfz}{\mbf{z}}

\newcommand{\G}{\mathbb{G}}

\nc{\mcal}[1]{{\mathcal #1}}
\nc{\calA}{\mcal{A}}\nc{\calB}{\mcal{B}}\nc{\calC}{\mcal{C}}\nc{\calD}{\mcal{D}}
\nc{\calE}{\mcal{E}} \nc{\calF}{\mcal{F}}\nc{\calG}{\mcal{G}}\nc{\calH}{\mcal{H}}
\nc{\calI}{\mcal{I}}\nc{\calJ}{\mcal{J}}\nc{\calK}{\mcal{K}}\nc{\calL}{\mcal{L}}
\nc{\calM}{\mcal{M}}\nc{\calN}{\mcal{N}}\nc{\calO}{\mcal{O}}\nc{\calP}{\mcal{P}}
\nc{\calQ}{\mcal{Q}}\nc{\calR}{\mcal{R}}\nc{\calS}{\mcal{S}}\nc{\calT}{\mcal{T}}
\nc{\calU}{\mcal{U}}\nc{\calV}{\mcal{V}}\nc{\calW}{\mcal{W}}\nc{\calX}{\mcal{X}}
\nc{\calY}{\mcal{Y}}\nc{\calZ}{\mcal{Z}}
\nc{\fA}{\frak{A}}\nc{\fB}{\frak{B}}\nc{\fC}{\frak{C}} \nc{\fD}{\frak{D}}
\nc{\fE}{\frak{E}}\nc{\fF}{\frak{F}}\nc{\fG}{\frak{G}}\nc{\fH}{\frak{H}}
\nc{\fI}{\frak{I}}\nc{\fJ}{\frak{J}}\nc{\fK}{\frak{K}}\nc{\fL}{\frak{L}}
\nc{\fM}{\frak{M}}\nc{\fN}{\frak{N}}\nc{\fO}{\frak{O}}\nc{\fP}{\frak{P}}
\nc{\fQ}{\frak{Q}}\nc{\fR}{\frak{R}}\nc{\fS}{\frak{S}}\nc{\fT}{\frak{T}}
\nc{\fU}{\frak{U}}\nc{\fV}{\frak{V}}\nc{\fW}{\frak{W}}\nc{\fX}{\frak{X}}
\nc{\fY}{\frak{Y}}\nc{\fZ}{\frak{Z}}
\nc{\fa}{\frak{a}}\nc{\fb}{\frak{b}}\nc{\fc}{\frak{c}} \nc{\fd}{\frak{d}}
\nc{\fe}{\frak{e}}\nc{\fFf}{\frak{f}}\nc{\fg}{\frak{g}}\nc{\fh}{\frak{h}}
\nc{\fri}{\frak{i}}\nc{\fj}{\frak{j}}\nc{\fk}{\frak{k}}\nc{\fl}{\frak{l}}
\nc{\fm}{\frak{m}}\nc{\fn}{\frak{n}}\nc{\fo}{\frak{o}}\nc{\fp}{\frak{p}}
\nc{\fq}{\frak{q}}\nc{\fr}{\frak{r}}\nc{\fs}{\frak{s}}\nc{\ft}{\frak{t}}
\nc{\fu}{\frak{u}}\nc{\fv}{\frak{v}}\nc{\fw}{\frak{w}}\nc{\fx}{\frak{x}}
\nc{\fy}{\frak{y}}\nc{\fz}{\frak{z}}

\newtheorem{theorem}{Theorem}[section]
\newtheorem{lemma}[theorem]{Lemma}
\newtheorem{corollary}[theorem]{Corollary}
\newtheorem{prop}[theorem]{Proposition}

\newtheorem{ax}{A}

\theoremstyle{definition}
\newtheorem{definition}[theorem]{Definition}
\newtheorem{example}[theorem]{Example}
\newtheorem{remark}[theorem]{Remark}

\newtheorem{assumption}[theorem]{Assumption}

\newtheorem{thm}{Theorem}

 \DeclareMathOperator{\id}{id}
 
\DeclareMathOperator{\ch}{CH} 
 
\DeclareMathOperator{\Hom}{{Hom}}

\DeclareMathOperator{\proj}{pr} 
\DeclareMathOperator{\Spec}{{Spec}} 

 \DeclareMathOperator{\End}{End}

\DeclareMathOperator{\Gm}{\bbG_m}
\DeclareMathOperator{\SL}{SL}

\DeclareMathOperator{\Laz}{{\mathbb{L}}}
\DeclareMathOperator{\MGL}{\Omega}

\DeclareMathOperator{\Gr}{Gr}

\DeclareMathOperator{\Td}{Td}
\DeclareMathOperator{\Sm}{Sm}
\DeclareMathOperator{\Ch}{Ch}

\newcommand{\lieg}{\fg}
\newcommand{\sO}{\mathcal{O}}

\DeclareMathOperator{\Spf}{Spf}

\DeclareMathOperator{\Res}{Res}

\newcommand{\inj}{\hookrightarrow}

\def\angl#1{{\langle #1\rangle}}

\newcommand{\pt}{\text{pt}}
\newcommand{\Aff}{\bbA}

\newcommand{\PP}{\bbP}
\newcommand{\ZZ}{\bbZ}
\newcommand{\QQ}{\bbQ}
\newcommand{\Z}{\bbZ}

\newcommand{\Q}{\bbQ}

\newcommand{\lbr}{[\hspace{-1.5pt}[}  
\newcommand{\rbr}{]\hspace{-1.5pt}]}  
\newcommand{\La}{\Lambda}
\newcommand{\la}{\lambda}
\newcommand{\al}{\alpha}
\newcommand{\be}{\beta}
\newcommand{\om}{\omega}
\newcommand{\de}{\delta}
\newcommand{\ga}{\gamma}
\newcommand{\Ga}{\Gamma}
\newcommand{\ka}{\kappa}
\newcommand{\diag}{\triangle}

\newcommand{\ep}{\epsilon}

\newcommand{\ttt}{\mathfrak{t}}   
\DeclareMathOperator{\aff}{{aff}}  
\newcommand{\FAL}{R\lbr\La\rbr_F} 
\newcommand{\FA}{R\lbr\Ga\oplus\La\rbr_F}  
\newcommand{\Dem}{\Delta}   
\newcommand{\HF}{\mathbf{H}_F}  
\newcommand{\IF}{\mathcal{I}_F}  

\newcommand{\DF}{\mathbf{D}_F} 

\newcommand{\orH}{A_H}  

\newcount\cols
{\catcode`,=\active\catcode`|=\active
 \gdef\Young(#1){\hbox{$\vcenter
 {\mathcode`,="8000\mathcode`|="8000
  \def,{\global\advance\cols by 1 &}%
  \def|{\cr
        \multispan{\the\cols}\hrulefill\cr
        &\global\cols=2 }%
  \offinterlineskip\everycr{}\tabskip=0pt
  \dimen0=\ht\strutbox \advance\dimen0 by \dp\strutbox
  \halign
   {\vrule height \ht\strutbox depth \dp\strutbox##
    &&\hbox to \dimen0{\hss$##$\hss}\vrule\cr
    \noalign{\hrule}&\global\cols=2 #1\crcr
    \multispan{\the\cols}\hrulefill\cr%
   }
 }$}}
}

\title[The formal affine Hecke algebra]
{Geometric representations of the formal affine Hecke algebra}
\date{\today}

\author[G.~Zhao]{Gufang~Zhao}
\address{Max-Planck-Institut f\"ur Mathematik,
Vivatsgasse 7,
53111 Bonn,
Germany}
\curraddr{Institut de Math\'ematiques de Jussieu, UMR 7586 du
CNRS, Batiment Sophie Germain, 75205 Paris Cedex 13, France}
\email{gufangzhao@zju.edu.cn}

\author[C.~Zhong]{Changlong~Zhong}
\address{University of Alberta, 632 CAB, Edmonton, AB T6G 2G1, Canada}
\email{zhongusc@gmail.com}


\begin{document}
\begin{abstract}
For any formal group law, there is a formal affine Hecke algebra defined by Hoffnung--Malag\'on-L\'opez--Savage--Zainoulline. Coming from this formal group law, there is also an oriented cohomology theory.
We identify the formal affine Hecke algebra with a convolution algebra coming from the oriented cohomology theory applied to the Steinberg variety. As a consequence, this algebra acts on the corresponding cohomology of the Springer fibers. This generalizes the  action of classical affine Hecke algebra on the $K$-theory of the Springer fibers constructed by Lusztig. We also give a  residue interpretation of the formal affine Hecke algebra, which generalizes the residue construction of Ginzburg--Kapranov--Vasserot when the formal group law comes from a 1-dimensional algebraic group.
\end{abstract}

\maketitle
\tableofcontents

\section{Introduction}
Affine Hecke-type algebras arise from the study of representations of Chevalley groups over $\fp$-adic fields, and their representations have been studied extensively in the past decades. In particular,  the classification of irreducible representations and the character formulas have been achieved in \cite{KL}, \cite{Gin85}, \cite{Gr94}, etc. The  study uses in an essential way a convolution construction of the affine Hecke algebra using equivariant $K$-theory of Steinberg variety in  \cite{Lus85}.

Let $G$ be a semi-simple, simply-connected linear algebraic group over an algebraically closed field $k$. Consider the adjoint action of $G$ on its Lie algebra $\lieg$. The set of nilpotent elements in $\lieg$ has a variety structure, called the nil-cone, denoted by $\calN$. The $G$-action on $\calN$ has  finitely many orbits. The variety $\calN$ is singular but admits a natural resolution of singularities, called the Springer resolution, as follows: Let $\calB$ be the complete flag variety, parametrizing the set of  Borel subalgebras of $\lieg$. Let  $\widetilde{\calN}$ be its cotangent bundle $T^*\calB$. 
Note that $T^*\calB$ can be alternatively interpreted as the variety of pairs $(\fb,x)$, where $\fb$ is a Borel subalgebra of $\lieg$, and $x$ is a nilpotent element in $\fb$. There is a natural map $\widetilde{\calN}\to \calN$, sending each pair $(\fb,x)$ to $x$, which gives the resolution. For any $x\in\calN$, the fiber of this map over $x$ is called the {\em Springer fiber}, denoted by $\calB_x$. Note that the Springer fibers are equi-dimensional, projective varieties, but in general  not smooth. For any two different points in the same $G$-orbit of $\calN$, the fibers are isomorphic. It is well-known that  representations of the Weyl group of $G$ can be constructed by looking at  Borel-Moore homology of  Springer fibers. 
This construction can be traced back to Springer, and later on was re-described and generalized by many others. In particular, Lusztig in \cite{Lus85} constructed an action of the affine Hecke algebra on equivariant $K$-theory of  Springer fibers.

More precisely, in the constructions of Springer and Lusztig, they identified respectively the group algebra of the Weyl group and the affine Hecke algebra as the top Borel-Moore homology and respectively the equivariant $K$-theory of the Steinberg variety $Z:=\widetilde \calN\times_\calN \widetilde \calN$, both endowed with  convolution products. 
The essential property used about  Borel-Moore homology and $K$-theory is that they both admit push-forwards for proper morphisms, and pull-backs for smooth morphisms. In fact, a functor from the category of smooth quasi-projective varieties to the category of commutative (graded) rings, that  admits these two properties together with  certain  compatibility conditions, is called an {\em oriented cohomology theory}.
Examples of  oriented cohomology theories include the Chow ring $\ch^*$ (see \cite{Ful}), the $K$-theory, the elliptic cohomologies, and the algebraic cobordism theory $\MGL^*$ of Levine and Morel (see \cite{LM}). 
It is proved by Levine and Morel that the algebraic cobordism $\MGL^*$ is the universal  oriented cohomology theory.
We refer  to \cite{LM} and \cite{PaninSmirnovII} for details. We will briefly recall the relevant notions and properties in Section~\ref{sec:prelim_coh}. Also in Section~\ref{sec:prelim_coh}, we will explain in details how to obtain a convolution algebra out of an oriented cohomology theory, as well as how to get natural representations of convolution algebras. 

For any oriented cohomology theory $A$, there is an associated formal group law $(R,F)$  where $R=A(\pt)$ is a commutative  ring called the coefficient ring, and $F(u,v)\in R[\![u,v]\!]$. For example, the formal group law associated to $\ch^*$ is $(\ZZ,u+v)$, the one associated to $K$-theory is $(\ZZ[\beta^\pm],u+v-\beta uv)$, and the one associated to $\MGL^*$ is the universal formal group law of Lazard, whose coefficient ring is  called the {\em Lazard ring}, denoted by $\Laz$.

The idea of using formal group laws to study generalized (equivariant) oriented cohomology theories  of flag varieties were first carried out by Bressler--Evens in \cite{BE}. Recently, the basic properties have been studied in explicit forms in the frame work of the so-called formal affine Demazure algebras in \cite{HMSZ}, \cite{CZZ1}, \cite{CPZ}, \cite{CZZ2},  and \cite{CZZ3}.
It is natural to ask, whether there is an convolution algebra acting on the cobordism group or elliptic cohomology group of Springer fibers, by considering convolutions with  corresponding oriented cohomology classes on the Steinberg variety. In the current paper, we answer this question. 

Starting with any formal group law $(R,F)$, in \cite{HMSZ}, Hoffnung, Malag\'on-L\'opez, Savage, and Zainoulline constructed a candidate of the algebra, called the {\it formal affine Hecke algebra}, denoted by $\HF$. This algebra could potentially act on the corresponding oriented cohomology of Springer fibers. The algebra $\HF$ is an $R$-algebra generated by the character group $\Lambda$ of the maximal torus in $G$, a formal variable $x_\gamma$, and elements $J^F_\alpha$ for  simple roots $\alpha$ (see Remark~\ref{rmk:gen_J} for details). We will show in this paper, that the algebra $\HF$, with slight modification from the original definition of \cite{HMSZ}, indeed acts on the corresponding oriented cohomology of the Springer fibers. 
For this purpose, we need to interpret $\HF$ as a convolution algebra of the Steinberg variety. This is the main theorem of this paper. 

\begin{thm}[Theorem~\ref{thm:main}, Theorem~\ref{thm:isom_HF_cob}, and Proposition~\ref{prop:action_springer_fiber}]\label{thm:intr}
Let $G$ be a semi-simple simply connected algebraic group. Let $A$ be an equivariant oriented cohomology theory for any smooth linear algebraic group $H$,  whose associated formal group law is $(R,F)$ where $R=A(\pt)$. Suppose $(R,F)$ together with the associated root system of $G$ satisfies Assumption \ref{assump:strong}. We have the following.
\begin{enumerate}\item There is an $R$-algebra homomorphism  $\Psi_A:\HF\to \End_R(A_{G\times\Gm}(\widetilde{\calN})).$
\item For any simple root $\alpha$, there is a cohomology class $J^A_\alpha\in A_{G\times\Gm}(Z_\alpha)$ on the irreducible component $Z_\alpha$ of the Steinberg variety labelled by $\alpha$, such that $\Psi_A$ sends $J^F_{\alpha}\in \HF$  to the convolution with $J_\alpha^A$.
\item\label{main-intro-3} If the equivariant oriented cohomology theory $A$ satisfies Assumption~\ref{assum:running}, then $\Psi_A$  induces an isomorphism $\HF\cong A_{G\times\Gm}(Z)$, where $A_{G\times\Gm}(Z)$ is endowed with the convolution product;
\item  Suppose $A$ is as in (\ref{main-intro-3}), then for any $x\in \calN$, the $R$-module $A^*(\calB_x)$ admits a natural action of the algebra $\HF$.
\end{enumerate}
\end{thm}

In~(\ref{main-intro-3}), for singular variety $X$,  the definition $A_{G\times\Gm}(X)$ under Assumption~\ref{assum:running} is explained in \S~\ref{subsec:coh_cob}.

The properties of the algebra $\HF$ are studied in Sections~\ref{sec:FAHA}-\ref{sec:furtherProp}. In particular, we give a presentation of this algebra by generators and relations, and show that the action of this algebra on the equivariant oriented cohomology of $\widetilde{\calN}$ is faithful.  We also prove the PBW property, i.e., there is a filtration on this algebra whose associated graded algebra is isomorphic to the degenerate affine Hecke algebra. The relations between the formal affine Hecke algebras associated to different formal group laws are also studied. In particular, if the coefficient ring $R$ contains $\Q$,  the corresponding formal affine Hecke algebra is isomorphic to the degenerate affine Hecke algebra (after completion). This isomorphism, using Theorem~\ref{thm:intr}, is identified with the bivariant Riemann-Roch functor of Panin and Smirnov. We would like to emphasize that for two isomorphic formal group laws, the corresponding Hecke algebras are isomorphic. However, the essential operators in the formal affine Hecke algebra, called the Demazure-Lusztig operators, depend on the formal group law itself rather than its isomorphism class.

It worth mentioning that we made adjustment to the original definition of the formal affine Hecke algebra of \cite{HMSZ}, not only for the purpose of convolution construction.\footnote{After the first version of this paper was announced, Kirill Zainoulline kindly pointed out to us that the definition of  $\HF$ in our paper was also considered  during the preparation of \cite{HMSZ}.  In the final version of \cite{HMSZ}, another definition was used in order to get better combinatorial properties. The latter definition is not explicitly involved in this paper, but it serves as an inspiration in finding the present form of the geometric construction.} In fact,  when $(R,F)$ is the elliptic formal group law, the corresponding elliptic affine Hecke algebra has been considered by Ginzburg--Kapranov--Vasserot in \cite{GKV97} from a totally different approach. In Section~\ref{sec: struc_res}, we prove a structure theorem of the formal affine Hecke algebra and interpret it as the residue-vanishing conditions of \cite{GKV97}.  The definition of the Demazure-Lusztig operators we give matches with the one in \cite{GKV97}. In other words, the aim of the current paper is to show the unity among \cite{Lus85}, \cite{GKV97}, and \cite{HMSZ}. For general Kac-Moody root systems, the unity between \cite{GKV97} and \cite{HMSZ} is also known, and will appear in \cite{CZZ}. 

When the oriented cohomology theory is the $K$-theory, the definition of equivariant $K$-theory we are using differs from the one used in \cite{Lus85} by a completion.  When considering the representations coming from Springer fibers, the classical Hecke algebra becomes interesting when the equivariant parameters are specialized to points other than the identity. Similarly, when $A$ is the elliptic cohomology theory,  the representations become interesting when the equivariant parameters are specialized to arbitrary points on the elliptic curve. 
On the other hand, through our formal approach, the equivariant parameters acts nilpotently on the Springer fibers, hence the only way to specialize them is to let them act trivially. 
There is a delocalized equivariant elliptic cohomology, postulated in \cite{Gr}, axiomized in \cite{GKV95}, and constructed in details by works of many people, a far from being complete list of which including \cite{And03}, \cite{Chen},  \cite{Gep}, and \cite{Lur}.
In fact, the elliptic affine Hecke algebra of \cite{GKV97} defined using the residue-vanishing construction is isomorphic to the equivariant elliptic cohomology of the Steinberg variety. Hence, its representations can be constructed from Springer fibers, as is sketched in \cite{Gr}. 
In order to make the results in the current paper available to general oriented cohomology theory, we postpone the detailed study of geometric representations of the elliptic affine Hecke algebra, following the outlines in \cite{Gr}, to a future publication \cite{ZZ}.

Also, we point out that there are many convolution constructions used in geometric representation theory. For example, the representations of the affine quantum groups (resp. the Yangians) are studied by considering the equivariant $K$-theory (resp. the Borel-Moore homology) of the quiver varieties in  \cite{Nak99} and \cite{Va00}. A Riemann-Roch type theorem, relating the completion of the Yangian and the quantum loop algebra, is proved in \cite{GTL}. It has been expected that (see e.g., \cite{Gr} and \cite{GKV95}) the elliptic quantum group of Felder in \cite{Fed} acts on the equivariant elliptic cohomology of quiver varieties.  This statement will be verified in \cite{Zha}. One can ask, for an arbitrary oriented cohomology theory, what the algebra is that acts on the corresponding equivariant oriented cohomology of the quiver varieties. This question has been studied by Yang and the first named author in  \cite{YZ}.

\subsection*{Acknowledgments}
The first named author is grateful to  Roman Bezrukavnikov and  Valerio Toledano Laredo for encouragements, and to  Valerio Toledano Laredo for introducing him to \cite{GKV97} and \cite{HMSZ} so that this joint project became possible. The authors would like to thank  Baptiste Calm\`es, Marc Levine, and Kirill Zainoulline for helpful discussions. The second named author is supported by PIMS and NSERC grants of Stefan Gille and Vladimir Chernousov.  This paper is prepared when both authors are hosted by the Max Plank Institute for Mathematics in Bonn. 

\section{Classical affine Hecke algebra}
In this section we briefly recall some basic notions of classical Hecke-type algebras.
\subsection{Root datum}
We fix the notation of root data following  \cite{CZZ1}. A {\it root datum} is an embedding $\Sigma \hookrightarrow \La^\vee, \al\mapsto \al^\vee$ of a non-empty subset $\Sigma$ of a lattice $\La$ into its dual lattice, satisfying certain conditions. The elements in $\Sigma$ are called {\it roots}, and the sub-lattice of $\La$ generated by $\Sigma$ is called the {\it root lattice}, denoted by $\La_r$. The dimension  of $\La_\Q:=\La\otimes_\Z\Q$ is called the {\em rank} of the root datum. The set $\La_w=\{\om\in \La_\Q\mid \angl{\om,\al^\vee}\in \Z\text{ for all }\al\in \Sigma\}$ is called the {\it weight lattice}. A root datum is called {\it irreducible} if it is not a direct sum of root data of smaller ranks, and it is called semi-simple if $\La_r\otimes_\Z\Q=\La_\Q$. From now on we always assume that the root datum is semi-simple. Note that in this case we have $\La_r\subseteq \La\subseteq \La_w$.

Denote $[n]=\{1,...,n\}$ where $n$ is the rank of the root lattice. The root  lattice has a basis $\Phi=\{\al_1,...,\al_n\}$ such that any $\al\in \Sigma$ is a $\Z$-linear combination of $\al_i$'s with either all positive or all negative coefficients. So there is a decomposition $\Sigma=\Sigma^+\sqcup \Sigma^-$. We call $\Sigma^+$ (resp. $\Sigma^-$) the set of {\it positive} (resp. {\it negative)} roots. We call the set $\{\om_i\}_{i\in [n]}$ such that $\angl{\om_j,\al_i^\vee}=\de_{ij}$ for $i,j\in [n]$ the set of fundamental weights, which is a basis of $\La_w$. 
Expressing $\Phi$ in terms of linear combinations of fundamental weights, the coefficients matrix is called the Cartan matrix of the root datum. The root datum is called {\it simply-connected} (resp. {\it  adjoint}) if $\La=\La_w$ (resp. $\La=\La_r$), and it is denoted by $\calD_n^{sc}$ (resp. $\calD_n^{ad}$), where $\calD=A,\dots,G$ is one of the Dynkin types. Let $\ttt$ be the torsion index of the associated simply-connected root datum defined in  \cite[\S5]{Dem}.

The reflection $\La\to \La:\la\mapsto \la-\angl{\la,\al^\vee}\al$ defined by $\al\in \Sigma$ is called a simple reflection, denoted by $s_\al$. The group $W$ generated by all simple reflections is called the Weyl group. It is also generated by $s_i:=s_{\al_i}$, $i\in [n]$. For $i\neq j\in [n]$, we have $(s_is_j)^{m_{ij}}=1$ for $m_{ij}\in \{2,3,4,6\}$. We can define the Bruhat order $(W,\le)$. The length of $w$ is denoted by $\ell(w)$. Let $w=s_{i_1}\cdots s_{i_k}$ be a reduced decomposition, then we define
\[
\Sigma(w)=\{\al_{i_1}, s_{i_1}(\al_{i_2}), \cdots ,s_{i_1}(\cdots (s_{i_{k-1}}(\al_{i_k}))\cdots)\}=w\Sigma^-\cap \Sigma^+.\]
\subsection{Classical Hecke-type algebras} There are many well-studied Hecke-type algebras.
\begin{example}The {\it Hecke algebra} $H$ associated to the Weyl group $W$ is the $\Z[q,q^{-1}]$-algebra with unit, generated by elements $T_i, i\in [n]$ subject to the relations
\[
(T_i+1)(T_i-q)=0,\]\[ ~\underbrace{T_{i}T_jT_i\cdots}_{m_{ij} \text{ times}} -\underbrace{T_jT_iT_j\cdots}_{m_{ij} \text{ times}}=0, ~\text{ if }(s_is_j)^{m_{ij} }=1.
\]
\end{example}
\begin{example}\label{ex:Haff}The {\it affine Hecke algebra} $H_{\aff}$ associated to $W$ is the algebra $\Z[q,q^{-1}][\La] \otimes_{\Z[q,q^{-1}]}H$, with the factors $H$ and $\Z[q,q^{-1}][\La]$ being subalgebras, and the relations between the two tensor factors being
\[
T_ie^{s_i(\la)}-e^\la T_i=(1-q)\frac{e^\la-e^{s_i(\la)}}{1-e^{-\al_i}}.
\]
For any $w=s_{i_1}\cdots s_{i_k}\in W$, define $T_w=T_{i_1}\cdots T_{i_k}$. The element $T_w$ in $H_{\aff}$ depends only on $w$ itself, not the decomposition. As a left $\Z[q,q^{-1}][\La]$-module, $H_{\aff}$ has a basis $\{T_w\}_{w\in W}$. Let $\iota:\Z[q,q^{-1}][\La]\to \Z[q,q^{-1}][\La]$ be the automorphism mapping $e^\la$ to $e^{-\la}$ and $q\to q$, then by \cite[Proposition 7.6.38]{CG}, the actions on $\Z[q,q^{-1}][\La]$ of $T_i$ and $u\in \Z[q,q^{-1}][\La]$ are 
\begin{equation}\label{eq:Haffact}
T_i:e^\la\mapsto \frac{e^\la-e^{s_i(\la)}}{e^{\al_i}-1}-q\frac{e^\la-e^{s_i(\la)+\al_i}}{e^{\al_i}-1}, \text{ and } u:e^\la\mapsto \iota(u)e^\la.\end{equation}
Note that the action of $\Z[q,q^{-1}][\La]$ on itself is a  multiplication  involving the map $\iota$.
\end{example}
\begin{example}
The {\it degenerate affine Hecke algebra} $H_{deg}$ associated with the Weyl group is  $ S_{\Z[\ep]}^*(\La)\otimes_{\Z[\ep]}\Z[\ep][W]$ as a $\Z[\ep]$-module,  where both $\Z[\ep][W]$ and $S^*_{\Z[\ep]}(\La)$ are subalgebras, and 
\begin{equation}\label{eq:Hdeg}
\theta_i\la-s_i(\la)\theta_i=\ep \angl{\la,\al_i^\vee}, ~i\in [n], ~\la\in \La.
\end{equation} Here $\theta_i$ is the element in $\Z[\ep][W]$ corresponding to $s_i\in W$. The algebra $H_{deg}$ has a natural grading by $\deg(\theta_i)=0$, $\deg (\ep)=1$ and $\deg (\la)=1$  for any $\la\in \La$. For each $w=s_{i_1}\cdots s_{i_k}\in W$, the element $\theta_w:=\theta_{i_1}\cdots \theta_{i_k}$ depends only on $w$. As a left $S^*_{\Z[\ep]}(\La)$-module, there is a basis $\{\theta_w\}_{w\in W}$. The action of $H_{deg}$ on $S^*_{\Z[\ep]}(\La)$ is 
\begin{equation}
\theta_i: \la\mapsto \ep\frac{\la-s_i(\la)}{\al_i}+s_i(\la), \text{ and } \mu:\la\mapsto \mu\la, ~\mu, \la\in \La.
\end{equation}
\end{example}

\section{Definition of the formal affine Hecke algebra}\label{sec:FAHA}
In this section we define the formal affine Hecke algebra, which was first  introduced in \cite{HMSZ}. We will use a slightly modified version, suited for our geometric purpose but still maintain all the algebraic properties.

\subsection{Formal group laws}
Let $R$ be a commutative ring, and let $(R,F)$ be a formal group law over $R$, that is, $F(x,y)\in R\lbr x, y\rbr$ such that 
\[
F(x,y)=F(y,x), ~F(x,0)=0, F(x,F(y,z))=F(F(x,y),z).
\]
 Let $-_Fx\in R\lbr x\rbr$ be the power series such that $F(x,-_Fx)=0$, and write $x+_Fy=F(x,y)$ and $x-_Fy=x+_F(-_Fy)$. Denote $\ka^F(x)=\frac{1}{x}+\frac{1}{-_Fx}$ and $\mu^F(x)=\frac{-_Fx}{-x}$.
\begin{lemma} \label{lem:fglkappa}
With notations as above, we have
\begin{enumerate}
\item \cite[Lemma 4.3]{HMSZ} $\ka^F(x)=0$ if and only if $-_Fx=-x$, if and only if $F(x,y)=(x+y)h(x,y)$ for some power series $h(x,y)$;
\item $\ka^F(x)=1$ if and only if $-_Fx=\frac{x}{x-1}$.
\end{enumerate}
\end{lemma}
The proofs are straightforward.

For simplicity, we will refer to the two cases in Lemma \ref{lem:fglkappa} as $\ka^F=0$ and $ \ka^F=1$, respectively. If $\ka^F=1$, it is called a normal formal group law in the terminologies of \cite{Zh13}. For normal formal group laws, a lot of notations in this paper can be simplified.

\begin{lemma}\label{lem:fgl1} For any formal group law $F$,  $\frac{x-_Fy}{x-y}$ is invertible in $R\lbr x,y\rbr$.
\end{lemma}
\begin{proof}We use the idea from the proof of \cite[Lemma 9.1]{CPZ}. Write $x-_Fy=\sum_{i=0}^\infty h_i(y)x^i$ with $h_i\in R\lbr y\rbr$, then $x-_Fy=(x-_Fy)-(y-_Fy)=\sum_{i=1}h_i(y)(x^i-y^i)$, hence $(x-y)$ divides $(x-_Fy)$. Moreover,  $\frac{x-_Fy}{x-y}$ is invertible if and only if $h_1(0)$ is invertible in $R$. Write $x+_Fy=x+y+xyg(x,y)$ with $g(x,y)\in R\lbr x,y\rbr$, and denote $-_Fy=y\tilde g(y)$ for some $\widetilde g(y)\in R\lbr y\rbr$, then $x-_Fy=x+_F(-_Fy)=x+y\tilde g(y)+xy\tilde g(y)g(x,y)$, so $h_1(0)=1$. 
\end{proof}

\begin{example} \label{ex:fgl}
\begin{enumerate}
\item For any commutative  graded ring $R$, the element $F_a(x,y)=x+y$ in $R[\![x,y]\!]$ defines the {\it additive formal group law}. We have $-_{F_a}x=-x$, $\ka^{F_a}=0$, and $\mu^{F_a}(x)=1$.
\item For any commutative  graded ring $R$, the element $F_m(x,y)=x+y-\be xy$ in $R[\![x,y]\!]$ with $ \be\in R$ defines a {\it multiplicative formal group law}. We have $-_{F_m}x=\frac{x}{\be x -1}$, $\ka^{F_m}=\be$ and $\mu^{F_m}(x)=\frac{1}{1-\beta x}$.
\item A {\it Lorentz formal group law} is defined by $F_l(x,y)=\frac{x+y}{1-\be xy}$ with $\be\in R$. We have $-_{F_l}x=-x$, $\ka^{F_l}=0$ and $\mu^{F_l}=1$.
\item The {\it elliptic formal group laws}: Given a family of elliptic curves over some ring $R$, there is a formal group law over the ring $R$, coming from the additive structure of the elliptic curves and a choice of parameter along the zero section over $\Spec R$. More precisely, fixing a local uniformizer $t$ around the identity section, the expansion of the group law of $E$ in terms of $t$ gives a formal group law $F_E$ with coefficients in $R$. 
For example,  in most of the cases one can choose the local parameter to be $\fl=\frac{\wp'}{\wp}$ where $\wp$ is the Weierstrass $\wp$-function, or $\fl=\sin$ where $\sin$ is the Jacobi sine-function. The corresponding formal group law is determined by $F(u,v)=\fl(\fl^{-1}(u)+\fl^{-1}(v))$. Note that for both choices, one has $\fl(-t)=-\fl(t)$.
\item There is a {\it universal formal group law} $(\Laz, F_{\Laz})$, whose coefficient ring $\Laz$, called the Lazard ring, is a polynomial ring in countably many generators over $\ZZ$. For any formal group law $(R, F)$, there exists a unique ring homomorphism $\phi_F: \Laz \to R$ such that $F=\phi_F(F_{\Laz})$.
\end{enumerate}
\end{example}

\subsection{Formal group algebra}\label{subsec:fga}
Let $\La$ be a free abelian group of rank $n$. Let $R\lbr x_\La\rbr$ be the ring of power series in variables $\{x_\la|\la\in \La\}$. Define the {\it formal group algebra} $R\lbr\La\rbr_F$ to be the quotient 
\[
R\lbr x_\La\rbr/\left (x_0, x_\la+_Fx_\mu-x_{\la+\mu}\right),
\]
where $x_0$ is the element determined by $0\in \La$. Let $\IF$ be the kernel of the augmentation map $R\lbr\La\rbr_F \to R, x_\la\mapsto 0$, then there is a filtration on $R\lbr\La\rbr_F$:
\[
R\lbr\La\rbr_F =\IF^0\supsetneq \IF^1\supsetneq \IF^2\supsetneq \cdots.
\]
The associated graded ring $\Gr R\lbr\La\rbr_F:=\bigoplus_{i\ge 0}\IF^i/\IF^{i+1}$ is isomorphic to the symmetric algebra $S_R^*(\La)$.  Indeed, if $\{\om_1,...,\om_n\}$ is a basis of $\La$, then $R\lbr \La\rbr_F\cong R\lbr \om_1,...,\om_n\rbr$.

\begin{example}
\begin{enumerate}
\item If $F=F_a$, then $R\lbr \La \rbr_{F_a}\cong S_R^*(\La)^\wedge, x_\la\mapsto \la$, where $S_R^*(\La)^\wedge$ is the completion of the symmetric algebra  at the augmentation ideal $\ker(S_R^*(\La)\to R, \la\mapsto 0)$.
\item If $F(x,y)=x+y-xy$, then we have an isomorphism $R\lbr \La \rbr_{F}\cong R[\La]^\wedge, x_\la\mapsto 1-e^{-\la}$, where $R[\La]^\wedge$ is  the completion of the group ring at the augmentation ideal 
$\ker (R[\La] \to R, e^\la\mapsto 1)$.
\end{enumerate}
\end{example}

The definition of $\FAL$ is functorial with respect to homomorphisms of  free abelian groups $\La\to \La'$, homomorphisms of  coefficient rings $R\to R'$,  and homomorphisms of formal group laws $h:(R,F)\to (R,F')$ over $R$.  We explain the latter two functoriality properties in details here.

If $(R,F)$ is a formal group law and  $f:R\to R'$ is a morphism of rings, then $f$ induces a formal group law $f(F)$ with coefficients in $R'$, by applying $f$ to the coefficients of $F$ one-by-one. Then the natural map $R\lbr \La\rbr_F\to R'\lbr \La \rbr_{f(F)}, x_\la\mapsto x_\la$ is a well-defined algebra homomorphism. Moreover, $R\lbr \La\rbr_F\hat{\otimes}_R R'\cong R'\lbr\La\rbr_{f(F)}$, where $\hat{\otimes}$ is the completed tensor in the following sense. In $R\lbr\La\rbr_F$, the power of ideals $\IF^i$ defines a filtration. It induces a filtration on the usual extension by scalars $R\lbr \La\rbr_F\otimes_R R'$.  The completed tensor product, denoted by $R\lbr \La\rbr_F\hat{\otimes}_R R'$, is the completion of the ring $R\lbr \La\rbr_F{\otimes}_R R'$ with respect to this filtration.

Following \cite[\S2.5]{CPZ}, a homomorphism of formal group laws $h:(R,F)\to (R,F')$ is a power series $h(x)\in R\lbr x\rbr$  such that 
\begin{equation}\label{eqn:morph_fgl}
h(x+_Fy)=h(x)+_{F'}h(y).
\end{equation}
It induces a ring homomorphism $\phi_{F', F}:R\lbr\La\rbr_{F'}\to \FAL$ defined by $\phi_{F', F}(x_\la)=h(x_\la)$ and extended by linearity and multiplicativity. That is, if $f(x)$ is any power series, then $\phi_{F',F}(f(x_\la)):=f(h(x_\la))$. It is well defined since 
\[
\phi_{F',F}(x_{\la+\mu})=h(x_{\la+\mu})=h(x_\la+_Fx_\mu)=h(x_\la)+_{F'}h(x_\mu)=\phi_{F',F}(x_\la)+_{F'}\phi_{F',F}(x_\mu)=\phi_{F',F}(x_\la+_{F'}x_\mu).
\]

If $R$ contains $\Q$, then any formal group law $(R,F)$ is isomorphic to $(R,F_a)$. More precisely, for any $(R,F)$ such that $R$ is a $ \Q$-algebra, according to \cite[Ch. IV, \S1]{Fr}, there exists an invertible $\fl(x)\in R\lbr x\rbr$  such that $\fl(\fl^{-1}(x))=\fl^{-1}(\fl(x))=x$ and $\fl:(R,F_a)\to (R,F)$ is an isomorphism of formal group laws. So $\fl(x+y)=\fl(x)+_F\fl(y)$, and it induces 
\[
\phi_{F,F_a}: \FAL \to R\lbr\La\rbr_{F_a}, ~x_\la\mapsto \fl(x_\la).
\]
Clearly $\fl^{-1}$ induces the inverse of $\phi_{F,F_a}$, and therefore$\phi_{F,F_a}$ is an isomorphism.
\subsection{Twisted formal group algebra} 
From now on we always fix a root datum $\Sigma\hookrightarrow \La^\vee$.
\begin{definition}We say that   $R\lbr\La\rbr_F$ is $\Sigma${\it -regular} if $x_\al$ is regular in $R\lbr\La\rbr_F$ for any $\al\in \Sigma$.
\end{definition}
By \cite[Definition 4.4]{CZZ1}, $R\lbr\La\rbr_F$ is $\Sigma$-regular if $2$ is regular in $R$, or if the root datum does not contain an irreducible component $C_k^{sc}, k\ge 1$. In this paper we  assume that $R\lbr\La\rbr_F$ is always $\Sigma$-regular. Sometimes we may assume the following stronger condition, which is essential due to \cite[Lemma 3.5]{CZZ1}, and is used in the proof of the structure theorem (Theorem \ref{thm:structHecke}) of the formal affine Hecke algebra. 
\begin{assumption}\label{assump:strong}We assume that the torsion index $\ttt$ is regular, and in addition $2$ is invertible in $R$ if the root datum  contains an irreducible component of type $C_l^{sc}, l\ge 1$.
\end{assumption}

We will need the following assumption in Lemma \ref{lem:div} and Theorem \ref{thm:residueDH}:
\begin{assumption} \label{assump:div} For each irreducible component of the root datum, assume that the corresponding integers or formal integers in Table \ref{tab:div} are regular in $R$ or $R\lbr x\rbr$, and  that $2$ is invertible in $R$ if the root datum contains a component of type $C_l^{sc}, l\ge 1$.
\end{assumption}

{\small
\newcommand{\doublecell}[2][c]{\begin{tabular}[#1]{@{}c@{}}#2\end{tabular}}
\newcommand{\doublecelltop}[2][t]{\begin{tabular}[#1]{@{}c@{}}#2\end{tabular}}
\begin{table}
\begin{tabular}{l||c|c|c|c|c|c|c|c|c|}
Type & \doublecelltop{$A_l$ \\ $(l\geq 2)$} & \doublecelltop{$B_l$ \\ $(l\geq 3)$} & \doublecelltop{$C_l$ \\ $(l\ge 2)$} & \doublecelltop{$D_l$ \\ $(l\geq 4)$} & $G_2$ & $F_4$ & $E_6$ & $E_7$ & $E_8$ \\
adjoint & $\emptyset$ & $2\cdot_F$ & $2\cdot_F$ & $\emptyset$ & \doublecell{$2\cdot_F$ \\ and $3\cdot_F$} & $2\cdot_F$ & $\emptyset$ & \doublecell{$2\cdot_F$ \\ or $3\cdot_F$} & \doublecell{$2\cdot_F$ \\ or $3\cdot_F$} \\
non adjoint & $|\La/\La_r|$ & $2$ & $2 \in R^\times$ & $2$ & - & - & $3$ & $2$ & - \\  
\end{tabular}
\medskip
\caption{Integers and formal integers assumed to be regular in $R$ or $R\lbr x \rbr$ in Assumption \ref{assump:div}.}\label{tab:div} 
\end{table}
}
\begin{lemma}\cite[Lemma 2.7]{CZZ2}\label{lem:div} Under Assumption \ref{assump:div}, $R\lbr \La\rbr_F$ is $\Sigma$-regular. For any two positive roots $\al\neq \be$ and $x\in R\lbr \La\rbr_F$, $x_\al|x_\be x$  implies  $x_\al|x.$
\end{lemma}

Let $\Ga$ be a free abelian group of rank 1 generated by $\ga\in \Ga$. We consider the algebra 
\[
S=S_F:=\FA\cong R\lbr \Ga\rbr_F\lbr \La\rbr_F.\] 
The action of the Weyl group $W$ on $\La$ induces an action of $W$ on $S$. More precisely, $w(x_\ga)=x_\ga$ and $w(x_\la)=x_{w(\la)}$ for any $ \la\in \La$. For any $\al\in \Sigma$, we define 
\[
\ka^F_\al=\frac{1}{x_\al}+\frac{1}{x_{-\al}} ~\text{ and }~\mu^F_\al=\frac{x_{-\al}}{-x_\al}
\]
in $S$. The {\it formal Demazure operator} and the {\it push-pull operator} are defined to be, respectively
\[
\Dem^F_\al(u)=\frac{u-s_\al(u)}{x_\al}\hbox{, and } ~C^F_\al(u)=\ka^F_\al u-\Dem^F_\al(u)=\frac{u}{x_{-\al}}+\frac{s_\al(u)}{x_\al}\hbox{, for any} ~u\in S.
\]
They are $R$-linear operators on $S$. We will skip the superscript $F$ if there is no confusion, and denote $\ka_i=\ka_{\al_i}, x_i=x_{\al_i}, x_{-i}=x_{-\al_i}, \Dem_i=\Dem_{\al_i}$ and $C_i=C_{\al_i}$. For any sequence $I=(i_1,...,i_k)$ with $i_j\in [n]$, denote $|I|=k$ and define
\[
\Dem_I(u)=\Dem_{i_1}\circ \cdots \circ \Dem_{i_k}(u), ~u\in S.
\]
We say that $I_w=(i_1,...,i_k)$ is a {\em reduced sequence} of $w$ if $w=s_{i_1}\cdots s_{i_k}$ is a reduced decomposition.   Unless $F$ is the additive formal group law or a multiplicative formal group law, the definition of $\Dem_{I_w}$ depends on $I_w$ (\cite[Theorem 3.7]{BE}).

\begin{remark}The operators $\Dem_\al$ and $C_\al$ are historically  considered by Demazure for the first time in \cite{Dem} in the study of the Chow rings of flag varieties. In geometry, the operator $C_\al$ is the composition of  push-forward and pull-back defined by $G/B\to G/P_{\al}$, where $B$ is a Borel subgroup and $P_\al$ is the minimal parabolic subgroup determined by $\al\in \Sigma$.
\end{remark}

 Let $Q^F=S[\frac{1}{x_\al}|\al\in \Sigma^+]$. We define the {\it twisted formal group algebra} $Q^F_W=Q^F\rtimes_{R}R[W]$, that is, it is isomorphic to $Q^F\otimes_RR[W]$ as an $R$-module, and the multiplication is given by 
\[
u\otimes\de_w \cdot u'\otimes\de_{w'}=uw(u')\otimes\de_{ww'}, ~u,u'\in Q^F, w,w'\in W.
\]
Note that the product is not commutative and $Q^F_W$ is not a $Q^F$-algebra as the embedding $Q^F\to Q^F_W, u\mapsto u\de_e$ is not central. Here $e\in W$ is the identity element, and we will denote $1=1\otimes\de_e$. Let  $Q_\al=S[\frac{1}{x_\be}\mid \be\in \Sigma^+, \be\neq \al]$, and consequently $S=\cap_{\al\in \Sigma^+}Q_\al$. 

For each root $\al$, we define the {\it formal Demazure element} and the {\it formal push-pull element} by 
\[
X^F_\al:=\frac{1}{x_\al}-\frac{1}{x_{\al}}\de_{s_\al}, ~Y^F_\al:=\ka_\al-X^F_\al=\frac{1}{x_{-\al}}+\frac{1}{x_\al}\de_{s_\al}.
\]
We also define
\begin{equation}\label{eq:T}
T^F_\al=x_\ga X_\al+\de_\al\in Q_W^F,
\end{equation}
and call it a {\em Demazure-Lusztig element (operator)}. For simplicity, we will sometimes skip the superscript $F$, and use the following notations for short: $\de_\al=\de_{s_\al}$, $\de_{i}=\de_{\al_i}$, $X_i=X_{\al_i}$, $X_{-i}=X_{-\al_i}$,  $Y_i=Y_{\al_i}$, $Y_{-i}=Y_{-\al_i}$, $T_i=T_{\al_i}$ and $T_{-i}=T_{-\al_i}$.

The following lemma follows from direct computation.
\begin{lemma}\label{lem:XYTproperty}
The operators satisfying the following relations:
\begin{enumerate}
\item $\de_w X_\al \de_{w^{-1}}=X_{w(\al)}$, $\de_w Y_\al \de_{w^{-1}}=Y_{w(\al)}$ and $\de_wT_\al \de_{w^{-1}}=T_{w(\al)}$ for  $w\in W$;
\item $X_\al \de_\al=-X_\al, \de_\al X_\al=X_\al+\ka_\al\de_\al-\ka_\al.$
\item $X_\al^2=\ka_\al X_\al, Y_\al^2=\ka_\al Y_\al$;
\item $X_\al u -s_\al(u)X_\al=\Dem_\al(u)$, $Y_\al u-s_\al(u)Y_\al=\frac{-1}{\mu_\al}\Dem_\al(u)=\Dem_{-\al}(u)$ for $u\in Q^F$;
\item $T_\al u-s_\al(u)T_\al=x_\ga \Dem_\al(u)$ for $u\in Q^F$;
\item $(T_\al-1)(T_\al+1-x_\ga \ka_\al)=0$. Moreover, $T_\al^2=x_\ga \ka_\al T_\al+1-x_\ga \ka_\al$, $T_\al^{-1}=\frac{T-x_\ga \ka_\al}{1-x_\ga \ka_\al}$. In particular, if $\ka^F=0$, then $T_\al^2=1$,  and $T_\al^{-1}=T_\al$. If $\ka^F=1$, then $T_\al^2=x_\ga T_\al+1-x_\ga$ and $T^{-1}_\al=\frac{x_{-\ga}}{x_\ga}T_\al+x_{-\ga}$.
\end{enumerate}
\end{lemma}
\begin{proof} We spell out the proof for $X_\al^2=\ka_\al X_\al$ only, as the rest are proved similarly. 
\[
X_\al ^2=(\frac{1}{x_\al}-\frac{1}{x_\al}\de_\al)(\frac{1}{x_\al}-\frac{1}{x_\al}\de_\al)=\frac{1}{x_\al^2}-\frac{1}{x_\al x_{-\al}}\de_\al-\frac{1}{x_\al^2}\de_\al+\frac{1}{x_\al x_{-\al}}=(\frac{1}{x_\al}+\frac{1}{x_{-\al}})(\frac{1}{x_\al}-\frac{1}{x_\al}\de_\al)=\ka_\al X_\al.
\]
\end{proof}
For any sequence $I=(i_1,...,i_l)$, we define 
\[
T_I=T_{i_1}T_{i_2}\cdots T_{i_l}.
\]
The operators $X_I,Y_I$ for any sequence $I=(i_1,...,i_l)$ are defined similarly. Let  $w\in W$ be any element and let $I_w=(i_1,...,i_l)$ be a reduced sequence of $w$, i.e., $w=s_{i_1}s_{i_2}\cdots s_{i_l}$ is a reduced decomposition. The operators $X_{I_w}, Y_{I_w}$ depend on the choice of $I_w$, not only on $w$ itself, unless $F$ is the additive formal group law or a multiplicative formal group law (\cite[Theorem 3.7]{BE}). Similar conclusion holds for $T_{I_w}$ (Remark \ref{rem:indIw}).

Define 
\begin{equation}\label{eq:tildec}
c_w=\prod_{\al\in \Sigma(w)}x_\al, \quad \tilde c_w=\prod_{\al\in \Sigma(w)}(x_\al-x_\ga).
\end{equation}
\begin{lemma}\label{lem:deltaT} Fix a set of reduced sequences $\{I_w\}_{w\in W}$, we have 
\[
T_{I_w}=\sum_{v\le w}b^T_{w,v}\de_v, ~\de_w=\sum_{v\le w}a^T_{w,v}T_{I_v}, ~b^T_{w,v}\in Q^F, ~a^T_{w,v}\in S[\frac{1}{\tilde c_{w_0}}],
\]
such that  $a_{w,w}^T=\frac{c_w}{\tilde c_w}$, $b^T_{w,w}=\frac{\tilde c_w}{c_w}$. Consequently, $\{T_{I_w}\}_{w\in W}$ is a set of linearly independent elements.
\end{lemma}
\begin{proof} This follows similarly as in Lemma 5.4 and Corollary 5.6 of \cite{CZZ1}.
\end{proof}
\begin{lemma} \label{lem:braid} If in $W$ the elements $s_i$ and $s_j$ satisfy $(s_is_j)^{m_{ij}}=1$, then the operators $T_i$ and $T_j$ on $S$ satisfy the following relation
\[
\underbrace{T_iT_jT_i\cdots}_{m_{ij} \text{ times }}-\underbrace{T_jT_iT_j\cdots}_{m_{ij} \text{ times} }=\sum_{\ell(v)\le m_{ij}-2} u_v^{ij}T_{I_v}, ~u^{ij}_v\in Q^F.
\]
\end{lemma}
\begin{proof} The proof is similar to that of \cite[Proposition 6.8.(a)]{HMSZ}.
\end{proof}
For any $i\neq j$, denote $x_{i+j}=x_{\al_i+\al_j}$, and define 
\[
\ka_{ij}=\frac{1}{x_{i+j}}(\frac{1}{x_j}-\frac{1}{x_{-i}})-\frac{1}{x_ix_j}.
\]
By \cite[Lemma 6.7]{HMSZ}, the elements $\ka_{ij}$ are $S$. Computation shows the following theorem, which gives an explicit expression of how the braid relation is deformed according to the formal group law.

\begin{theorem} \label{thm:braid}
With notations as above, the operators $T_i$ and $T_j$ on $S$ satisfy the following relations
\begin{enumerate}
\item $T_iT_j=T_jT_i$, if $(s_is_j)^2=1$ for $s_i$ and $s_j\in W$;
\item $T_iT_jT_i-T_jT_iT_j=x_\ga^2\left(\ka_{ji}T_j-\ka_{ij}T_i+\frac{\ka_j-\ka_i}{x_{i+j}}\right)$, if $(s_is_j)^3=1$ for $s_i$ and $s_j\in W$. 
\item $\underbrace{T_iT_jT_i\cdots}_{m_{ij} \text{ times }}-\underbrace{T_jT_iT_j\cdots}_{m_{ij} \text{ times} }=\sum_{\ell(v)\le m_{ij}-2} u_v^{ij}T_{I_v}$ with the coefficients $u_v^{ij}$ for $m_{ij}=2,3,4,6$ belong to $S$. 
\end{enumerate}
\end{theorem} 

Claim (3) follows from Corollary \ref{cor:struct2} below.

It follows from \cite[p. 809]{BE} that $\kappa_{ij}=0$ if and only if $F(x,y)=x+y-\beta xy$ for some $\beta\in R$, in which case we also have $\ka_i=\ka_j$ (see Example \ref{ex:fgl}.(2)). In other words, in this case the braid relations  $T_iT_jT_i=T_jT_iT_j$ hold.
\begin{remark} \label{rem:indIw} Direct computation shows that the coefficients $u_v^{ij}=0$ for all $v\in W$ and $i,j\in [n]$ if and only if $F(x,y)=x+y-\beta xy$ for some $\beta\in R$. This generalizes \cite[p. 809]{BE}. In this case, the definition of $T_{I_w}$ does not depend on the choice of $I_w$.
\end{remark}

\subsection{Formal affine Hecke algebra}
Now we give the definition of the formal affine Hecke algebra associated to any formal group law $(R,F)$.

\begin{definition}Define the {\it formal affine Demazure algebra} $\DF$ to be the $R$-subalgebra of $Q^F_W$ generated by $S$ and $X_i, i\in [n]$. Define the {\it formal affine Hecke algebra} $\HF$ to be the $R$-subalgebra of $Q^F_W$ generated by $S$ and $T_i,i\in [n]$.
\end{definition}
\begin{remark}
It follows from Theorem~\ref{thm:struct}~(\ref{prop:structDem}) below that the definition of $\DF$ does not depend on the choice of simple roots. As a ring itself, $\HF$  does not either, but for different choices of simple roots $\Phi$, the embeddings of $\HF$ into $Q_W^F$  are different. As can be easily seen that  $\de_w\not\in \HF$,  and hence $T_\al\in \HF$ if and only if $\al\in \Phi$. For different choices of $\Phi$, the embeddings are related by the Weyl group action. Note also that our definition of $T_\al$ and $\HF$ is different from the corresponding ones in \cite{HMSZ}. 
\end{remark}
\begin{remark}
The $S$-dual of the formal affine Demazure algebra is the algebraic replacement of the $T$-equivariant algebraic oriented cohomology of flag varieties \cite{KK86}, \cite{KK90} and \cite{CZZ3}. On the other hand, for additive or multiplicative formal group laws, the formal affine Demazure algebra gives the classical affine nil-Hecke algebra and affine $0$-Hecke algebra, respectively. See \cite{HMSZ} for the details. 
\end{remark}

\begin{remark}\label{rmk:gen_J}
For any simple root $\alpha\in\Phi$, let $J_\alpha^F=(x_{-\al}-x_\ga)Y_{-\al}=T_{-\alpha}-\mu_{\al}-x_\ga \ka_\al\in \HF$. If $I_v$ is a reduced sequence of $v$, we can define $J^F_{I_v}$ correspondingly. Then clearly $\HF$ is also generated by $S$ and $J^F_\alpha, \al\in \Phi$.  This set of generators are used when relating $\HF$ with the convolution algebra $A_{G\times \Gm}(Z)$ in Section \ref{sec:conv_Hecke}.
\end{remark}

\section{The structure theorem of the formal affine Hecke algebra}\label{sec: struc_res}
In this section we prove the structure theorem of the formal affine Hecke algebra and some of the basic properties. We also show that our construction recovers the residue construction of affine Hecke algebras due to Ginzburg--Kapranov--Vasserot in \cite{GKV97}.

\subsection{The structure theorem}
The ring $Q_W^F$ acts on $Q^F$ by 
\[
(u\de_w)\cdot u'=uw(u'), ~ u,u'\in Q^F, w\in W.
\]
We have
\[
X_\al\cdot u=\frac{u-s_\al(u)}{x_\al}=\Dem_\al(u), ~Y_\al\cdot u=\frac{u}{x_{-\al}}+\frac{s_\al(u)}{x_\al}=C_\al(u), ~T_\al\cdot u=x_\ga \Dem_\al(u)+s_\al(u), ~u\in Q^F,
\]
so $X_\al\cdot S\subseteq S, Y_\al\cdot S\subseteq S$ and $ T_\al\cdot S\subseteq S$. Indeed, we have 
\begin{theorem}[The structure theorem]\label{thm:struct} Under  Assumption \ref{assump:strong}, we have
\begin{enumerate}
\item\label{prop:structDem} $\DF=\{z\in Q_W^F|z\cdot S\subseteq S\}$;
\item \label{thm:structHecke} 
$
\HF=\{z\in \sum_{w\in W}u_w T_{I_w}|u_w\in Q^F \text{ and }z\cdot S\subseteq S\}.
$ 
\end{enumerate}
\end{theorem}
We only show the proof of (\ref{thm:structHecke}), as (\ref{prop:structDem}) is proved in \cite[Proposition~6.2]{CZZ1}.
\begin{proof}We follow the method of \cite[\S4]{Zh13}. It suffices to show that the right hand side is contained in  $\HF$.  Firstly, we introduce some notations. Let $\tilde S=R\lbr\Ga\rbr_F[\frac{1}{\ttt x_\ga}]\lbr\La\rbr_F$, then $x_\ga-x_\al$ is invertible in $\tilde S$. We similarly define $\tilde{Q}^F=\tilde S[\frac{1}{x_\al}|\al\in \Sigma]\supsetneq Q^F$ and $\tilde{Q}^F_W=\tilde{Q}^F\rtimes_{R}R[W]$. Correspondingly, we define $\tilde X_\al$ and $\tilde T_\al$ in $\tilde{Q}^F_W$ similar as $X_\al$ and $T_\al$ in $Q^F_W$. Also, we define the operator 
\[
\tau_i:\tilde S\to \tilde S, ~u\mapsto x_\ga \Dem_i(u)+s_i(u), ~i\in [n].
\]
Let $\tilde{\IF}$ be the kernel of the augmentation map $\ep:\tilde S\to R\lbr\Ga\rbr_F[\frac{1}{\ttt x_\ga}]$ sending $ x_\la\mapsto 0$ for any $\la\in \La$, and let $\Gr \tilde S=\oplus_{j\ge 0}\tilde{\IF}^j/\tilde{\IF}^{j+1}$.  By \cite[Lemma 5.3]{CPZ}, we have $\Dem_i(\tilde{\IF}^j)\subseteq \tilde{\IF}^{j-1}$ and the operator $\Gr\Dem_i:\Gr\tilde S\to \Gr\tilde S$ is of degree $-1$. Since $\de_i(\tilde{\IF}^j)=\tilde{\IF}^j$, we also have $\tau_i(\tilde{\IF}^j)\subseteq \tilde{\IF}^{j-1}$ for any $j$. So $\tau_i$ induces an operator $\Gr\tau_i$ on $\Gr \tilde S$ of degree $-1$. Then similar as \cite[Lemma 4.7]{Zh13}, we have $\Gr\tau_I=x_\ga^{|I|}\Gr\Dem_I$. Moreover, $\tau_I(\tilde{\IF}^j)\subseteq \tilde{\IF}^{j-|I|}$. If $I$ is not reduced, then $\tau_I(\tilde{\IF}^j)\subseteq \tilde{\IF}^{j-|I|+1}$. 

We need the following Lemma. Let $w_0\in W$ be the longest element, and $I_0$ be a reduced sequence of $w_0$. 
\begin{lemma} \label{lem:invmat}There exists an element $u_0\in \tilde{\IF}^{\ell(w_0)}$ such that $\Dem_{I_0}(u_0)\equiv \ttt\mod \tilde{\IF}$. Moreover, if $|I|\le \ell(w_0)$, then 
\begin{equation}
\ep\tau_I(u_0)=\left\{\begin{array}{ll}x_\ga^{\ell(w_0)}\ttt, & \text{if }I \text{ is reduced and } |I|=\ell(w_0),\\
0 , & \text{otherwise.}\end{array}\right.
\end{equation}
Finally, the matrix $(\tau_{I_v}\tau_{I_w}(u_0))_{(v,w)\in W\times W}$ is invertible in $\tilde S.$
\end{lemma}
\begin{proof}The existence of $u_0$ follows from \cite[\S5.2]{CPZ} (the $a$ in loc.it.), and the computation of $\ep\tau_I(u_0)$ follows from Lemma 5.3.3 in {\it loc.it.}. The conclusion about the matrix $(\tau_{I_v}\tau_{I_w}(u_0))_{(v,w)\in W\times W}$ follows from the proof of Proposition 6.6 in {\it loc.it.}.
\end{proof}
Let $z=\sum_{w\in W}u_wT_{I_w}, u_w\in Q^F$  be such that $z\cdot S\subseteq S$, we show that $u_w$ itself belongs to $ S$ for any $w$. We can view $z$ as in $\tilde{Q}^F_W$. Applying $z$ to $\tau_{I_v}(u_0)\in \tilde S$ for all $v\in W$, we get a system of linear equations
\[
\sum_{w\in W}u_w\tau_{I_w}\tau_{I_v}(u_0)=z\cdot \tau_{I_v}(u_0)\in \tilde S.
\]
By Lemma \ref{lem:invmat} we know that the matrix $(\tau_{I_w}\tau_{I_v}(u_0))_{(w,v)\in W^2}$ is invertible, so $u_w\in \tilde S$. We consider the following diagram of embeddings:
\[
\xymatrix@R=20pt@C=20pt{S\ar@{=}[r] &R\lbr \Ga\rbr _F\lbr \La\rbr _F\ar@{^{(}->}[r] \ar@{^{(}->}[d]& R\lbr \Ga\rbr _F[\frac{1}{\ttt x_\ga}]\lbr \La\rbr _F \ar@{^{(}->}[d]\ar@{=}[r] & \tilde S \\
Q^F\ar@{=}[r]& R\lbr \Ga\rbr _F\lbr \La\rbr _F[\frac{1}{x_\Sigma}]\ar@{^{(}->}[r] &  R\lbr \Ga\rbr _F[\frac{1}{\ttt x_\ga}]\lbr \La\rbr _F[\frac{1}{x_\Sigma}] &,}\]
where $x_\Sigma:=\prod_{\al\in \Sigma}x_\al$.  By \cite[Lemma 3.3 and 3.5]{CZZ1}, we know that $u_w\in \tilde S\cap Q^F=S$. So $z\in \HF$.
\end{proof}
The following corollary generalizes \cite[Theorem 7.2.16]{CG}.
\begin{corollary}\label{cor:faithact}Under  Assumption \ref{assump:strong}, the action of $\HF$ on $S$ is faithful.
\end{corollary}
\begin{proof}Assume that $z=\sum_{w\in W}u_wT_{I_w}\in \HF$ with $u_w\in Q^F$ and $z\cdot S=0$, then the proof of Theorem \ref{thm:struct}(\ref{thm:structHecke}) implies that $u_w\in S$ for all $w\in W$. Moreover, $(\frac{1}{x_\al^j}z)\cdot S=0\subsetneq S$ for any $j\ge 1$ and $\al\in \Sigma$. By Theorem \ref{thm:struct}(\ref{thm:structHecke}), we have $\frac{1}{x^j_\al}z=\sum_{w\in W}\frac{u_w}{x^j_\al}T_{I_w}\in \HF$, that is, $\frac{u_w}{x_\al^j}\in S$ for any $j\ge 1$. This is impossible unless $u_w=0$.
\end{proof}
\begin{corollary}\label{cor:struct2} Suppose that $R\lbr\La\rbr_F$ is $\Sigma$-regular, then
\begin{enumerate}
\item  $\HF$ is the left $S$-module with basis $\{T_{I_w}\}_{w\in W}$.
\item Fix a set of reduced sequences $\{I_w\}_{w\in W}$, and let $I_v'$ is another reduced sequence of $v$, then $T_{I_v}-T_{I_v'}=\sum_{w<v}a_{w}T_{I_w}$ with $a_w\in S$.
\item $\HF$ is the $R$-subalgebra of $Q^F_W$ generated by $S$ and $T_i, i\in [n]$ subject to the relations in Lemma~\ref{lem:XYTproperty}.(5)-(6) and in Theorem~\ref{thm:braid}.
\end{enumerate}
\end{corollary}
\begin{proof}
If the assumption in Theorem \ref{thm:struct} is satisfied, then its proof implies that $\{T_{I_w}\}_{w\in W}$ is a basis of $\HF$ as a $S$-module. The general case follows from functoriality of the formal group algebra and $\HF$, similar as \cite[Corollary 4.12]{Zh13}. 

 Statements~(2)  and (3) follow from similar arguments as in \cite[Lemma 7.1]{CZZ1} and \cite[Theorem 6.14]{HMSZ}, respectively.
\end{proof}
\begin{remark}\label{rem:JFproperty} It is not difficult to see that $\{J^F_{I_w}\}_{w\in W}$ satisfies  the same properties in Corollary \ref{cor:struct2}.(1)-(2) after replacing $T$ by $J$.
\end{remark}

The following corollary generalizes \cite[Theorem 6.5]{Lus88} and \cite[Theorem 7.1.14]{CG}.
\begin{corollary}If $R$ contains $\Z$, then the center of $\HF$ is equal to $S^W$.
\end{corollary}
\begin{proof} In this case $R\lbr\La\rbr_F$ is $\Sigma$-regular,  and the conclusion follows similarly as \cite[Corollary 4.13]{Zh13}.
\end{proof}

\subsection{The residue construction of Ginzburg--Kaprankov--Vasserot}
In the remaining part of this section, we work under Assumption \ref{assump:div}. We interpret the structure theorems in terms of the residue-vanishing conditions in \cite{GKV97}.

\begin{definition} For any $z=\sum_{w\in W}a_w\de_w\in Q_W^F$ with $a_w\in Q^F$, consider the following conditions:
\begin{enumerate}
\item[\textbf{R}1]  for any root $\al$, $x_\al a_w\in Q_\al$;
\item[\textbf{R}2] for any root $\al$, $a_w+a_{s_\al w}\in Q_\al$;
\item[\textbf{R}3]  ${a_w}/{\tilde c_w}\in Q^F$, where $\tilde c_w$ was defined in \eqref{eq:tildec}.
\end{enumerate}
We define $\widetilde \DF$ to be the subset of $Q_W^F$ consisting of $z=\sum_{w\in W}a_w\de_w$ with $a_w\in Q^F$ satisfying \textbf{R}1 and \textbf{R}2, and define $\widetilde \HF\subsetneq \widetilde \DF$ consisting of  $z=\sum_{w\in W}a_w\de_w$, $a_w\in Q^F$ satisfying \textbf{R}3.
 
 \end{definition}
 
 These conditions we give here can be interpreted as conditions on the residues similar to \cite{GKV97}.  Indeed, for the formal group law $(R,F)$, we have the formal scheme $\Spf(S)$. Inside $\Spf(S)$, there are divisors defined by $x_\alpha$ for each root $\alpha$. Then elements 
 of $ Q^F$ can be re-interpreted as rational functions on $\Spf(S)$ which are regular away from the divisors defined by $x_\alpha$'s. Condition \textbf{R}1 is equivalent to saying that each $a_w$ has at most a pole of order 1 along the divisor defined by $x_\alpha$. Define the residue of $f$ at the divisor $x_\alpha=0$ to be $\Res_{x_\alpha}f:=\pi_\alpha(x_\alpha f)\in Q_\al/(x_\alpha)$ where $\pi_\alpha:Q_\al\to Q_\al/(x_\alpha)$ is the natural projection. Then condition \textbf{R}2 is equivalent to letting $\Res_{x_\alpha}a_w+\Res_{x_\alpha}a_{s_\alpha w}=0$ for any $w\in W$. 
 Note that although $\Res_{x_\alpha}$ defined this way depends on a choice of coordinates, the condition \textbf{R}2 is independent of coordinates.  Condition \textbf{R}3 is equivalent to letting each $a_w$ be vanishing along the divisor defined by $x_\alpha=x_\gamma$ for any $\al\in \Sigma(w)$. 
 
 It follows from similar argument as in \cite[Theorem 1.4]{GKV97} that $\widetilde \DF$  and $\widetilde \HF$ are algebras (note that this argument relies on Lemma \ref{lem:div}, hence we need to work under Assumption \ref{assump:div}).
 \begin{theorem}\label{thm:residueDH} Under Assumption~\ref{assump:div},  we have $\widetilde \DF=\DF$ and $\widetilde \HF=\HF$ as subalgebras of $Q_W^F$.
 \end{theorem}
 \begin{proof} Also it is easy to see that $X_i\in \widetilde \DF$ and $T_i\in \widetilde\HF$ for any $i$. Therefore, we have $\DF\subseteq\widetilde \DF$ and $\HF\subseteq \widetilde \HF$. 
 
 Let $z=\sum_{w\in W}a_w\de_w\in \widetilde \DF$, with $ a_w\in Q^F$. To show that $\widetilde \DF\subseteq \DF$,  we need to show that $z\cdot S\subseteq S$. For any root $\al$, let ${}^\al W=\{w\in W|\ell(s_\al w)>\ell(w)\}$, then $W=^\al W\sqcup s_\al\cdot {}^{\al}W$. For any $x\in S$, we have
 \begin{gather*}
 z\cdot x=\sum_{w\in W}a_ww(x)=\sum_{w\in {}^\al W}\left(a_{w}w(x)+a_{s_\al w}s_\al w(x)\right)\\
 =\sum_{w\in ^\al W}\left((a_w+a_{s_\al w})w(x)+a_{s_\al w}(s_\al w(x)-w(x))\right).
 \end{gather*}
 We know by \textbf{R}2 that  $a_w+a_{s_\al w}\in Q_\al$ , and $w(x) \in S$. Moreover, $s_\al w(x)-w(x)=x_\al \Dem_\al(w(x))$. By \textbf{R}1 we know that $x_\al a_{s_\al w}\in Q_\al$.  So $z\cdot x\in Q_\al$ for any $\al$, and hence $z\cdot x\in \cap _{\al\in \Sigma^+}Q_\al=S.$ In particular, $\widetilde \DF\subseteq \DF$ by Theorem~\ref{thm:struct}( \ref{prop:structDem}).

Now we prove that $\widetilde \HF\subseteq \HF$. If $z\in \widetilde \HF\subsetneq \widetilde \DF$, then from the first part we know that $\DF=\widetilde \DF$, so $z\cdot S\subseteq S$. By Theorem \ref{thm:struct}(\ref{thm:structHecke}), it suffices to  show that if $z$ satisfies Condition \textbf{R}3, then it is a $Q$-linear combination of $\{T_{I_w}\}_{w\in W}$. According to Lemma \ref{lem:deltaT}, we have
\[
z=\sum_{\ell(w)\le k}a_w\de_w=\sum_{\ell(w)\le k}a_w\sum_{v\le w}a^T_{w,v}T_{I_v}=\sum_{v\in W}(\sum_{v\ge w, \ell(w)\le k}a_wa^T_{w,v})T_{I_v}.
\]
Fix $v_0\in W$ such that $\ell(v_0)=k$, then 
\[b_{v_0}:=\sum_{w\le v_0, \ell(w)\le k}a_wa_{w,v_0}^T=a_{v_0}a^T_{v_0,v_0}=a_{v_0}\frac{c_{v_0}}{\tilde c_{v_0}}.\]
By  Condition \textbf{R}3, $\frac{a_{v_0}}{\tilde c_{v_0}}$ lies in $ Q^F$, hence  $b_{v_0}\in Q^F$. Note that $b_{v_0}T_{I_{v_0}}$ satisfies Condition \textbf{R}3, so it suffices to replace $z$ by $z':=z-b_{v_0}T_{I_{v_0}}=\sum_{\ell(w)\le k}a'_w\de_w$ which satisfies  $a_{v_0}'=0$. Repeating the above argument on $z'$, and by decreasing induction on $k$ we see that $z$  is a $Q^F$-linear combination of $\{T_{I_w}\}_{w\in W}$. This finishes the proof.
\end{proof}

\section{Further properties and examples of the formal affine Hecke algebra}\label{sec:furtherProp}
In this section we study  the formal affine Hecke algebra when the formal group law is $F_a$ or $F_m$, or when the ring $R$ contains the field of rational numbers. We also prove the PBW property, i.e., we construct a filtration on the formal affine Hecke algebra whose associated graded ring is isomorphic to the degenerated affine Hecke algebra.

\subsection{The cases of special formal group laws}
\begin{prop}Let $R=\Z$ and $F=F_m(x,y)=x+y-xy$. Define the map 
\[\Z[q,q^{-1}][\La]\hookrightarrow \Z[\Ga\oplus\La]^\wedge\cong \Z\lbr\Ga\oplus\La\rbr_{F_m}, ~q\mapsto \frac{1}{1-x_\ga}, ~e^{\la}\mapsto 1-x_{-\la}.\]
Then there is an isomorphism of rings
\[\psi:\mathbf{H}_{F_m}\overset{\sim}\longrightarrow \Z[\Ga\oplus\La]^\wedge\otimes_{\Z[q,q^{-1}][\La]}H_{\aff},\hbox{ with } ~\psi(T^{F_m}_{i})=1\otimes\frac{T_{-i}}{q}\hbox{ and }~\psi(x_\la)= x_{-\la}\otimes 1\hbox{ for }~\la\in \La,
\]
such that that following diagram commutes
\[
\xymatrix{\mathbf{H}_{F_m}\ar[r]^-{f_1}\ar[d]_-\psi& \End_{\Z}(\Z[\Ga\oplus\La]^\wedge)\\
   \Z[\Ga\oplus \La]^\wedge\otimes_{\Z[q,q^{-1}][\La]}H_{\aff} \ar[ru]_-{f_2} &.}
\]
Here $f_2$ is induced by the action of  $H_{\aff}$ on $\Z[q,q^{-1}][\La]$ (see \cite[Theorem 7.2.16]{CG}), and $f_1$ is induced by the action of $\mathbf{H}_{F_m}$ on $R\lbr\Ga\oplus\La\rbr_{F_m}$.
\end{prop}
\begin{proof} By Corollary \ref{cor:struct2}, the algebra $\HF$ is generated by $\Z\lbr\Ga\oplus\La\rbr_{F_m}$ and $\{T_i^{F_m}\}_{i\in [n]}$, subject to the braid relations (Remark \ref{rem:indIw}) and the relations (5)  and (6) in Lemma \ref{lem:XYTproperty}. Direct computation shows that (5) and (6) coincide with the corresponding defining relation of $H_{\aff}$ in Example~\ref{ex:Haff}. In particular, $\psi$ is an isomorphism. 

We also have that 
\[
f_1(\frac{1}{1-x_\ga}T_{-i}^{F_m}):e^\la\mapsto \frac{e^\la-e^{s_i(\la)}}{e^{\al_i}-1}-q\frac{e^\la-e^{s_i(\la)+\al_i}}{e^{\al_i}-1} \text{ and }f_1(x_\mu): e^\la\mapsto x_{\mu}e^\la.
\]
Therefore, by \eqref{eq:Haffact}, $f_1(\frac{1}{1-x_\ga}T_{-i}^{F_m})=f_2(T_i)=f_2\circ\psi(\frac{1}{1-x_\ga}T_{-i}^{F_m})$ and $f_1(x_\mu)=f_2(x_{-\mu})=f_2\circ \psi(x_\mu)$. So the diagram commutes.
\end{proof}
\begin{remark} Since the action of $\Z[q,q^{-1}][\La]\subsetneq H_{\aff}$ on $\Z[q,q^{-1}][\La]$ involves an automorphism $\iota$ (see Example \ref{ex:Haff}),  our definition of $\psi$ maps $x_\la$ to $x_{-\la}\otimes 1$ so that the above diagram commutes.

\end{remark}
\begin{prop}\label{prop:HFa}Let $F=F_a$, we define a map
\[
S^*_{\Z[\ep]}(\La)\to \Z\lbr\Ga\oplus \La\rbr_{F_a}, ~\ep\mapsto x_{\ga}, ~\la\mapsto x_{\la}, 
\]
then there is an isomorphism of rings
\[
\mathbf{H}_{F_a}\cong \Z\lbr\Ga\oplus\La\rbr_{F_a}\otimes_{S^*_{\Z[\ep]}(\La)}H_{deg}, ~T^{F_a}_i\mapsto 1\otimes \theta_i, ~x_\la\mapsto x_\la\otimes 1,
\]
which intertwines the actions of these two rings on $\Z\lbr\Ga\oplus\La\rbr_{F_a}$.
\end{prop}
\begin{proof}
Identifying $\ep$ with $x_\ga$ and $\la$ with $x_\la$ respectively, we have 
\[
T^{F_a}_i: \la\mapsto \ga\frac{\la-x_i(\la)}{\al_i}+s_i(\la)=\theta_i\cdot \la.
\]
Moreover, by Remark \ref{rem:indIw} and Corollary~\ref{cor:struct2}, the algebra $\mathbf{H}_{F_a}$ is generated by $\Z\lbr\Ga\oplus\La\rbr_{F_a}$ and $\{T^{F_a}_i\}_{i\in [n]}$  subject to the same relations as $H_{deg} $. So the two rings are isomorphic. The fact that the isomorphism commutes with the actions follows from direct computation.
\end{proof}

\begin{example} Let $F_l(x,y)=\frac{x+y}{1+\beta xy}$ be the Lorentz formal group law, then $-_{F_l}x=-x, \ka^{F_l}=0$ and $\ka^{F_l}_{ij}=\beta$. We have
\begin{enumerate}
\item $(T_i^{F_l})^2=T_i^{F_l}$,
\item If $m_{ij}=2$, then $T_i^{F_l}T_j^{F_l}=T_j^{F_l}T_i^{F_l}$.
\item If $m_{ij}=3$, then $T_i^{F_l}T_j^{F_l}T_i^{F_l}-T_j^{F_l}T_i^{F_l}T_j^{F_l}=\beta x_\ga(T_j-T_i)$.
\end{enumerate}
If the root datum is simply laced, then these three relations together with Lemma \ref{lem:XYTproperty}.(5) form a complete set of defining relations of $\mathbf{H}_{F_l}$.
\end{example}

\subsection{The formal affine Hecke algebra with rational coefficients}\label{subsec:HFrational}
Suppose $R$ is a $\Q$-algebra. From  \S\ref{subsec:fga} we know that there is an isomorphism $\phi_{F,F_a}:R\lbr\Ga\oplus \La\rbr_F\to R\lbr\Ga\oplus\La\rbr_{F_a}$, which induces an isomorphism $Q^{F}_W\to Q^{F_a}_W$ by $u\de_w\mapsto \phi_{F,F_a}(u)\de_w$. 
\begin{lemma}\label{lem:fgl2}If $f(x)\in \Q\lbr x\rbr$ such that $f(0)=0$ and the coefficient of $x$ is nonzero, then $f(x)(x-y)$ divides $xf(y)-yf(x)$.
\end{lemma}
\begin{proof}
By assumption $\frac{f(x)}{x}$ is invertible in $\Q\lbr x\rbr$, so it suffices to show that $x(x-y)$ divides $g(x,y):=xf(y)-yf(x)$. Since $x$ and $x-y$ are coprime, it suffices to show that both $x$ and $x-y$ divide $g(x,y)$. But this follows since $g(0,x)=g(x,x)=0.$
\end{proof}
\begin{theorem}\label{prop:FAHAQIso}
Assume $R$ is a $\Q$-algebra, and $\fl(t)\in R[\![t]\!]$ is such that $F(x,y)=\fl^{-1}(\fl( x)+\fl (y))$. Then, there is an isomorphism
\[
\widetilde \phi_{F,F_a}:\HF\to \mathbf{H}_{F_a}, ~T^F_i\mapsto \frac{\fl(x_\ga)}{\fl(x_i)}+\frac{\fl(x_i)-\fl(x_\ga)}{\fl(x_i)}\de_i.
\]

\end{theorem}
The inverse $\widetilde\phi_{F_a,F}$ of $\widetilde\phi_{F,F_a}$ is determined by $\phi_{F_a,F}$ in \S~\ref{sec:FAHA}.
\begin{proof}Firstly, we  show that $\widetilde \phi_{F,F_a}(T_i^F)\in \mathbf{H}_{F_a}$. For simplicity, denote $\fl_\ga=\fl(x_\ga)$ and $\fl_i=\fl(x_i)$, then direct computation shows that
\[
\widetilde \psi_{F,F_a}(T_i^F)=\frac{\fl_\ga}{\fl_i}+\frac{\fl_i-\fl_\ga}{\fl_i}\de_i=\frac{x_i}{\fl_i}\frac{\fl_i-\fl_\la}{x_i-x_\la}T_i^{F_a}+\frac{\fl_\ga x_i
-x_\ga\fl_i}{\fl_i(x_i-x_\la)}.
\]
By \cite[Ch. IV, \S1]{Fr} we know that the power series $\fl(x)$ has zero constant term and the coefficient  of $x$ is 1. Therefore,  by Lemma  \ref{lem:fgl1} and Lemma \ref{lem:fgl2}, the expressions 
\[\frac{x_i(\fl_i-\fl_\la)}{\fl_i(x_i-x_\la)}~\text{ and }~\frac{\fl_\ga x_i
-x_\ga\fl_i}{\fl_i(x_i-x_\la)}
\] all belong to $R\lbr\Ga\oplus\La\rbr_{F_a}$. So $\widetilde \psi_{F,F_a}:\HF\to \mathbf{H}_{F_a}$ is well defined. Moreover, it is not difficult to see that the map $\widetilde \phi_{F_a,F}$ mapping $T^{F_a}_i$ to $\frac{\fl^{-1}(x_\ga)}{\fl^{-1}(x_{_i})}+\frac{\fl^{-1}(x_{_i})-\fl^{-1}(x_{\ga})}{\fl^{-1}(x_i)}\de_i\in Q^F_W$ is the inverse of $\widetilde \psi_{F,F_a}$, hence the conclusion follows. \end{proof}

\subsection{Filtration on the formal affine Hecke algebra}\label{subsec:PBW}
 Let $\IF\subset S$ be the kernel of the augmentation map $S\to R$,  $x_\la\mapsto 0$ for any $\la\in \Ga\oplus\La$, (note this is different from the $\tilde{\IF}$ in the proof of Theorem \ref{thm:structHecke}) and let $\Gr S=\oplus_{i\ge 0}\IF^i/\IF^{i+1}\cong S_R^*(\Ga\oplus\La)$. We define a filtration of $\HF$ as follows: $\HF^{j}\subseteq \HF$ is the $R$-submodule spanned by $uT_{I}$ such that $u\in \IF^j$. By \cite[Lemma 5.3]{CPZ} we know that $\Dem_i(\IF^j)\subseteq \IF^{j-1}$. Therefore, we have  $T_i\cdot \IF^j\subseteq \IF^j$. Therefore, $\HF^j$ is the set of $z\in \HF$ such that $z\cdot (\IF^k)\subset \IF^{j+k}$ for all $k\ge 0$. We have the following filtration 
\[
\HF=\HF^0\supsetneq \HF^1\supsetneq \HF^2\supsetneq \cdots.
\] 
\begin{lemma} \label{lem:filprod}
For any $j,j'\ge 0$, we have $\HF^j\cdot \HF^{j'}\subseteq \HF^{j+j'}.$
\end{lemma}
\begin{proof}By Lemma \ref{lem:XYTproperty}.(5), we know that $T_iu=s_i(u)T_i+x_\ga \Dem_i(u)$. Since $\Dem_i(\IF^j)\subseteq \IF^{j-1}$ and $s_i(\IF^j)=\IF^j$, so moving an element of $S$ from the right of $T_i$ to the left of $T_i$  will not change the degree of this element. Hence, if $\deg u=j$, $\deg u'=j'$, then $uT_Iu'T_{I'}\in \HF^{j+j'}$ for any $I,I'$, so the conclusion follows.
\end{proof}
 Denote $\HF^{(i)}=\HF^i/\HF^{i+1}$ and let $\Gr\HF=\oplus_{i\ge 0}\HF^{(i)}$, then by definition $\Gr\HF$ is automatically a $\Gr S$-module, and by Lemma \ref{lem:filprod} it is a ring as well. Moreover, $\HF^j\cdot \IF^k\subseteq\IF^{k+j}$. Therefore, $\Gr\HF$ acts on $\Gr S$.  The following theorem generalizes \cite[Proposition 12.3]{Gin}:
\begin{theorem} [The PBW property] \label{thm:filgrade} If $R\lbr\La\rbr_F$ is $\Sigma$-regular,  then there is an isomorphism of graded rings
\[
\Gr\HF\cong R\otimes _\Z H_{deg},
\]
which intertwines their actions on $\Gr S\cong S^*_R(\Ga\oplus\La)$  after identifying $\ep$ with $\la$ as in Proposition \ref{prop:HFa}.
\end{theorem}
\begin{proof}Let $\widetilde T_i$ denote the image of $T_i$ in $\Gr \HF$. Since $\HF$ is generated by $S$ and $T_i, i\in [n]$, the algebra $\Gr\HF$ is generated by $\Gr S$ and $\widetilde T_i, i\in [n]$, subject to the same relations (after reducing to the associated quotient). The relation in Lemma \ref{lem:XYTproperty}.(6) reduces to $\widetilde T_i^2=1$. For $u=x_\la$, the relation in Lemma \ref{lem:XYTproperty}.(5) reduces to $\widetilde T_i \la-s_i(\la)\widetilde T_i=\ga \widetilde\Dem_i(\la)$, where $\widetilde\Dem_i$ is the induced action of $X_\al$ on $\Gr S$. By \cite[Proposition 4.4]{CPZ} we know that $\widetilde\Dem_i(\la)=\Dem^{F_a}_i(\la)=\frac{\la-s_i(\la)}{\al_i}=\angl{\la, \al^\vee_i}$, hence $\widetilde T_i \la-s_i(\la)\widetilde T_i=\ga \angl{\la,\al_i^\vee}$.  By direct computation, all the coefficients $u_v^{ij}$ appeared  in \eqref{lem:braid} are divisible by $x_{\ga}^2$. Therefore, the right hand side of \eqref{lem:braid} belongs to $\HF^1$. But the left hand side belongs to $\HF^{0}$ and hence \eqref{lem:braid} induces braid relations, i.e., $(\widetilde T_i \widetilde T_j)^{m_{ij}}=1$. So we conclude that $\Gr\HF$ satisfies all the defining relations of $H_{deg}$, which means  $\Gr \HF\cong R\otimes _\Z H_{deg}$. It is straightforward to show that the action of $\tilde T_i$ coincide with that of $\theta_i$.
\end{proof}

\begin{example}\label{ex:HFdeg0}It is not difficult to see that in $\HF/(x_\ga)$, $T_i$ is identified with $\de_i$. Therefore we have $\HF/(x_\ga)\cong R\lbr\La\rbr_F\rtimes_RR[W].$ 
\end{example}

\section{Convolution in equivariant oriented cohomology theories}\label{sec:prelim_coh}
In this section we collect basic notions about equivariant oriented cohomology theory. 
For any equivariant oriented cohomology theory, we will also explain the construction of a convolution algebra and its representations.

\subsection{Equivariant oriented cohomology theories}
In this subsection, following the setting of \cite[\S2]{CZZ3}, we  recall the axioms defining the equivariant oriented cohomology theories, the examples of which are studied in \cite{Des}, \cite{EG}, \cite{HM},  \cite{Kr}, \cite{LM},  \cite{PaninSmirnovII}, \cite{Th} and \cite{To}. 

Let  $H$ be a linear algebraic group over $k$. Let $\Sm^H_k$ be the category of smooth quasi-projective varieties with an $H$-action, and the morphisms are also $H$-equivariant.  Let $\pt=\Spec (k)$ be the point. An equivariant oriented cohomology theory over $k$ is the following data: 
\begin{itemize}
\item[(a)] For any linear algebraic group $H$, a contravariant functor $\orH$ from $\Sm_k^H$ to the category of commutative rings with units;
\item[(b)] for each projective morphism $f:X\to Y$, there is a push-forward morphism $f_*:\orH(X)\to \orH(Y)$, satisfying the projection formula;
\item[(c)] for any morphism of smooth linear algebraic groups $H'\to H$, there is a natural transformation $\orH\to A_{H'}\circ Res$, where $Res:\Sm^H_k\to \Sm^{H'}_k $ is the canonical restriction; 
\item[(d)] a natural transformation $c^H:K_H\to \widetilde \orH$, where for any $X\in \Sm^H_k$, $K_H(X)$ is the Grothendieck group of $H$-equivariant vector bundles on $X$, and $\widetilde \orH(X)=\orH(X)\lbr t\rbr$ with the multiplicative structure. 
\end{itemize}
This data should satisfy conditions \textbf{A}1--\textbf{A}9 in \cite[\S2]{CZZ3}, which we recall as follows:
\begin{ax}[Compatibility for push-forwards] Push-forward is compatible with compositions of morphisms and commutes with pull-backs in transversal squares.
\end{ax}
\begin{ax}[Compatibility for restrictions]
Restriction is compatible with compositions of morphisms of groups and commutes with push-forwards.
\end{ax}
\begin{ax}[Localization sequence] For any smooth closed $H$-subvariety $i:Z\hookrightarrow X$  with $H$-equivariant open complement $j:U\to X$, the following sequence is exact:
\[
\orH(Z)\overset{i_*}\to \orH(X) \overset{j^*}\to \orH(U)\to 0.
\]
\end{ax}
\begin{ax}[Homotopy invariance] For any projection $p:X\times_k \bbA^n\to X$ in $\Sm^H_k$ such that $H$ acts on $\bbA^n$ linearly, the pull-back $p^*$ is an isomorphism.
\end{ax}
\begin{ax}[Normalization] For any regular embedding $i:D\to X$ of codimension 1 in $\Sm^H_k$, we have $c_1^H(\calO(D))=i_*(1)\in \orH(X)$.
\end{ax}
\begin{ax}[Torsors] Let $p:X\to Y$ be a morphism in $\Sm^H_k$. Let $H'\subset H$ be a closed normal subgroup such that $H'$ acts trivially on $Y$ and $p:X\to Y$ is an $H'$-torsor. Denote $q:H \to H/H'$, then the composition $A_{H/H'}(Y)\overset{Res_q}\longrightarrow A_H(Y)\overset{p^*}\longrightarrow A_H(X)$ is an isomorphism.
\end{ax}
\begin{ax}If $H=\{1\}$, $A_{1}$ is an oriented cohomology theory in the sense of \cite[Definition 1.1.2]{LM}.
\end{ax}
\begin{ax}[Self-intersection formula] Let $i:Y\to X$ be a regular embedding of codimension $d$ in $\Sm^H_k$, then the normal bundle $\calN_{Y/X}$ is $H$-equivariant and we have $i^*i_*(1)=c_d^H(\calN_{Y/X})$.
\end{ax}
\begin{ax}[Quillen formula]  If $\calL_1,\calL_2$ are two line bundles on $X$, we have 
\[
c_1(\calL_1\otimes \calL_2)=F(c_1(\calL_1), c_1(\calL_2))
\]
where $F$ is the formal group law over $A_{\{1\}}(\pt)$ associated to $A:=A_{\{1\}}$.
\end{ax}

 For simplicity, we denote $R=A(\pt)$, and $c_1=c_1^H$ for any group $H$. The equivariant Chow ring determines the additive formal group law $F_a$, the equivariant K-theory determines the multiplicative formal group law $F_m$, and the equivariant algebraic cobordism determines the universal formal group law $F_{\Laz}$ over the Lazard ring $\Laz$.

Following \cite[Definition 2.1]{CZZ3}, we also assume that the equivariant oriented cohomology is {\it Chern- complete over the point}, that is, $\orH(\pt)$ is separated and complete with respect to the $\gamma$-filtration defined by $c_1.$ Note that the examples of equivariant oriented cohomology theories constructed using Totaro's process are all Chern-complete over the point, for instance, the equivariant Chow ring $\ch_H$ by Edidin and Graham \cite{EG}, and the equivariant algebraic cobordism $\MGL_H$ in \cite{Des}, \cite{HM}, \cite{Kr}. The equivariant K-theory of Thomason \cite{Th} is not Chern-complete but it will be after completion at the $\gamma$-filtration. 

\begin{example}\label{ex:equ_cob_tens}
Let $k$ be a field of characteristic zero, and let $\Omega(-)$ be the algebraic cobordism of Levine-Morel.  Following \cite{To}, in \cite{HM} it has been explained how $\Omega$ extends  to an equivariant cohomology theory.  In particular, for any $X\in \Sm_k^H$, Totaro's construction produces a commutative ring $\Omega_H(X)$, together with a filtration on it.

For any formal group law $(R,F)$, one can define an equivariant oriented cohomology theory, sending any $H$-variety $X$ to $\Omega_H(X)\otimes_{\Laz}R$ where the map $\Laz\to R$ is the classifying map induced by the formal group law $F$.
The filtration on $\Omega_H(X)$ via Totaro's construction induces a filtration on $\Omega_H(X)\otimes_{\Laz}R$.
The completion of  $\Omega_H(X)\otimes_{\Laz}R$ with respect to this filtration is denoted by $\Omega_H(X)\hat{\otimes}_{\Laz}R$.
The pull-back, push-forward, and restriction of groups are defined in the naive way. It is clear that all the axioms of equivariant oriented cohomology theories are satisfied. In other words, any formal group law $(R,F)$ gives rise to an equivariant oriented cohomology theory in this way, which is automatically Chern complete. 
\end{example}

There is a nice formula for push-forwards  from the total space of a projective bundle to the base, which in the generality as stated below is due to Quillen. Specializing to the ordinary $K$-theory, this formula becomes the familiar Weyl character formula. Therefore, the general push-forward formula will be referred to as the {\em Quillen-Weyl character formula}.

\begin{lemma}\cite[Proposition 5.30]{Vish}
Let $A$ be an equivariant oriented cohomology theory with associated formal group law $(R,F)$. Let $X$ belong to $\Sm^H_k$, $V$ be some $n$-dimensional $H$-equivariant vector bundle on $X$, and $\pi: \PP_X(V) \to X$ be the corresponding projective bundle. Let $f(t) \in \orH(X)[\![t]\!]$. Then,
\begin{equation}\label{Quill-formula}
\pi_*(f(c_1(\sO(1))))=\sum_i\frac{f(-_{F}\lambda_i)}{\prod_{j\neq i}(\lambda_j-_{F}\lambda_i)},
\end{equation}
where $\lambda_i$ are $c_1$ of the Chern roots of $V$.
\end{lemma}

\subsection{Convolution algebra and its representations}\label{subsec:coh_cob}
Let $M_i$ be  smooth quasi-projective $H$-varieties for $i=1,2,3$, and let $Z_{12}\subseteq M_1\times M_2$ and $Z_{23}\subseteq M_2\times M_3$ be smooth $H$-stable closed subvarieties. Let $\proj_{i,j}:M_1\times M_2\times M_3\to M_i\times M_j$ be the projection. Denote $i:Z_{12}\times_{M_2}Z_{23}\to M_1\times M_2\times M_3$ to be the fiber product of $\proj_{12}^{-1}(Z_{12})\times_{(M_1\times M_2\times M_3)}\proj^{-1}(Z_{23})$; Assume $Z_{12}\times_{M_2}Z_{23}$ is smooth and assume $Z_{13}\subseteq M_1\times M_3$ contains the image of $Z_{12}\times_{M_2}Z_{23}$ under the projection $\proj_{13}:M_1\times M_2\times M_3\to M_1\times M_3$, and $p:Z_{12}\times_{M_2}Z_{23}\to Z_{13}$ is proper. Then we define $*:A_H(Z_{12})\otimes_{A_H(k)}A_H(Z_{23})\to A_H(Z_{13})$ to be \[\mu\otimes\mu'\mapsto p_* i^*(\proj_{12}^*\mu\cdot\proj_{23}^*\mu').\]

\begin{example}\label{ex:conv_sm}
We look at special cases that will be used.
\begin{enumerate}
\item Let $M_i\to M_0$ be projective equivariant morphisms of quasi-projective $H$-varieties with $M_i$ smooth for $i=1,2$. Let $i_Z:Z\to M_1\times_{M_0}M_2$ be an $H$-stable smooth closed subvariety, proper over $M$.
Take $M_3=\pt$ and $Z_{12}=Z$, $Z_{23}=M_2$, and $Z_{13}=M_1$.
Then for any $\eta\in A_H(Z)$ we have an operator $\eta*_Z\in \Hom_R(A_H(M_2),A_H(M_1))$ such that  $\eta*_Z X=\proj_{1*}\left((\proj_2^*\mu)\cdot \eta\right)\in A_H(M_1)$ for any $\mu\in A_H(M_2)$, where $\proj_i:Z\to M_i$ is the projection to the $i$-th coordinate.

\item Let $M_0=\Spec(k)$, $M_1=M_2=M_3=M$ be projective over $k$, and $Z_{12}=Z_{23}=Z_{13}=M\times M$. The above construction endows $A_H(M\times M)$ an associated product $*_{M\times M}$, called the convolution product. This product will also be denoted simply by $*$ if $M$ is clear from the context.

\item When $M$ is smooth and projective over $k$, there is a well-defined $R$-module homomorphism $A_H(M\times M)\to\End_R(M)$, sending $\eta$ to the operator $\eta*_{M\times M}$. 
\end{enumerate}
\end{example} 

We would like to define a convolution algebra structure on $A_H(M\times_{M_0}M)$, and make sense of its representation on $A(M_x)$ where $M_x$ is the fiber of the map $M\to M_0$ over $x\in M_0$. 
Unfortunately, the varieties $M\times_{M_0}M$ and $M_x$ are in general singular, although they could be smooth in special cases. For arbitrary equivariant oriented cohomology theory, {\it a priori} $A_H(M\times_{M_0}M)$ and $A(M_x)$ are not well-defined. 
We need to restrict ourselves onto some special type of oriented cohomology theories while needed. More precisely, sometimes we need to consider equivariant oriented Borel-Moore homology theories on the category of $H$-schemes of finite type over $k$, which restricts to equivariant oriented cohomology theories on $\Sm^H_k$. (See \cite[Definition~5.1.3]{LM} for the definition of oriented Borel-Moore homology theories  in the non-equivariant case.) 

Recall from \cite[Chapter~2]{LM} that  for any scheme $X$ of finite type (not necessarily smooth), the algebraic cobordism group $\MGL(X)$ is well-defined as a graded abelian group. It enjoys the following properties, proved in \cite[Theorem~7.1.1]{LM}.

\begin{prop}
Assume $k$ has characteristic zero.
\begin{enumerate}
\item For any two schemes $X$ and $Y$, there is a exterior product $\boxtimes:\Omega(X)\times\Omega(Y)\to \Omega(X\times Y)$.
\item For any smooth morphism $f:X\to Y$, or even more generally, local complete intersection morphism, the pull-back $f^*:\Omega(Y)\to \Omega(X)$ is well-defined; 
\item The pull-back has a refined version, i.e., for any local complete intersection morphism $f:X\to Y$ and an arbitrary morphism $Z\to Y$, there is a refined pull-back map $f^!:\Omega(Z\times_YX)\to \Omega(Z)$, which specializes to $f^*$ when $Z\to Y$ is the identity morphism on $Y$.
\item For proper morphism $f:X\to Y$, there is a push-forward $f_*:\Omega(X)\to \Omega(Y)$. Push-forwards is compatible with pull-backs in a sense spelled out in \cite[Definition~5.1.3]{LM}.
\item For any line bundle $\calL$ on $X$, there is a first Chern class operator $\widetilde{c_1}(\calL):\MGL(X)\to \MGL(X)$, which commutes with refined pull-backs. For any two line bundles $\calL$ and $\calM$, the operators $\widetilde{c_1}(\calL)$ and $\widetilde{c_1}(\calM)$ commute. 
\end{enumerate}
\end{prop}

When $X$ is smooth, the diagonal embedding $\Delta:X\to X\times X$ is a local complete intersection morphism. The ring structure on $\Omega(X)$ is obtained from $\Delta^*\circ\boxtimes:\Omega(X)\otimes \Omega(X)\to \Omega(X\times X)\to \Omega(X)$. 

Let $X$ be a smooth variety, so that $\Omega(X)$ has a commutative ring structure. Then for any closed subvarieties $Z_1$ and $Z_2$ of $X$, we can define \begin{equation}\label{eq:cap}\cap:\Omega(Z_1)\otimes_{\Omega(\pt)}\Omega(Z_2)\to \Omega(Z_1\times_XZ_2)
\end{equation} as follows. Let $\Delta:X\to X\times X$ be the diagonal embedding. As $X$ is smooth, this is a local complete intersection morphism. Let  $Z_1\times Z_2\to X\times X$ be the obvious map. Its base change with respect to $\Delta$ is the embedding $Z_1\times_XZ_2\to X$. We define $\cap$ to be the composition of $\boxtimes:\Omega(Z_1)\times\Omega(Z_2)\to \Omega(Z_1\times Z_2)$ and the refined pull-back $\Delta^!:\Omega(Z_1\times Z_2)\to \Omega(Z_1\times_XZ_2)$.

If $X$ is an $H$-scheme of finite type, we define $\Omega_H(X)$ using Totaro's construction in \cite{HM}. There are still well-defined pull-back for equivariant local complete intersection morphisms and push-forward for equivariant proper morphisms, and the first Chern class operator  $\widetilde c_1^H(\calL)$ on $\MGL_H(X)$ is still well-defined for $H$-equivariant line bundle $\calL$ over $X$.
Also, by construction,  $\widetilde c_1(\calL)$ is compatible with respect to restriction of groups. Consequently, for any smooth $H$-variety $X$, and $Z_1$, $Z_2\subseteq X$ two $H$-stable closed subvarities, $\cap:\Omega_H(Z_1)\otimes_{\Omega_H(\pt)}\Omega_H(Z_2)\to \Omega_H(Z_1\times_XZ_2)$ is well-defined.The details are spelled out in  \cite{HM}.

\begin{assumption}\label{assum:running}
We assume $k$ has characteristic zero, and assume the equivariant oriented cohomology theory $A$ satisfies that 
\[
A_H(X)=\MGL_H(X)\hat{\otimes}_{\Laz} R, \quad X\in \Sm_k^H
\]  where $(R,F)$ is the formal group law associated to the cohomology theory $A$. 
\end{assumption}
 
The completed tensor product on the right hand side is explained in Example~\ref{ex:equ_cob_tens}.
Note that this assumption is satisfied if  the underlying ordinary oriented cohomology theory $A$ satisfies that $A(X)=\Omega(X)\otimes_{\Laz} R$, and  the equivariant $A_H(X)$ is obtained from the ordinary theory by Totaro's construction.

For simplicity, in the rest of this section, we concentrate on oriented cohomology theories satisfying this assumption. Note that this family is in one-to-one correspondence with formal group laws, and it includes a lot of examples people have considered, e.g., the Chow ring $\ch$, the (completion of) $K$-theory, and obviously the algebraic cobordism theory itself. For any equivariant oriented cohomology theory satisfying Assumption~\ref{assum:running}, and for any singular $H$-variety $X$, we define $A_H(X)$ simply as $\Omega_H(X)\otimes_{\Laz} R$. The push-forwards and pull-backs are defined in the natural way.
 
Let $M_i$ be smooth quasi-projective $H$-varieties for $i=1,2,3$, and let $Z_{12}\subseteq M_1\times M_2$ and $Z_{23}\subseteq M_2\times M_3$ be $H$-stable closed subvarieties. Note that $Z_{12}$ and $Z_{23}$ are not necessarily smooth, which is different from the case in \S\ref{subsec:coh_cob}. Let $\proj_{i,j}:M_1\times M_2\times M_3\to M_i\times M_j$ be the projection. Denote $Z_{12}\times_{M_2}Z_{23}\to M_1\times M_2\times M_3$ to be the fiber product of $\proj_{12}^{-1}(Z_{12})\times_{(M_1\times M_2\times M_3)}\proj^{-1}(Z_{23})$; and define $Z_{12}\circ Z_{23}\subseteq M_1\times M_3$ to be the image of $Z_{12}\times_{M_2}Z_{23}$ under the projection $\proj_{13}:M_1\times M_2\times M_3\to M_1\times M_3$. Note that in general $Z_{12}\circ Z_{23}$ is not smooth even when both $Z_{12}$ and $Z_{23}$ are.

The following Proposition allows us to consider convolution algebras and their representations. The equivariant $K$-theory version is \cite[2.7.5]{CG}.

\begin{prop}\label{prop:conv_action}
Under Assumption~\ref{assum:running}.
\begin{enumerate}
\item Assume further that the natural map $Z_{12}\times_{M_2}Z_{23}\to Z_{12}\circ Z_{23}$ is proper. Then, there is a well-defined map \[*:A_H(Z_{12})\otimes_{\orH(\pt)} A_H(Z_{23})\to A_H(Z_{12}\circ Z_{23}).\] 
\item The map $*$ is associative in the usual sense. 
\item Suppose $M\to M_0$ is a projective morphism, and let $Z$ be an $H$-stable closed subvariety of $M\times_{M_0}M$. Then, there is a well defined morphism $A_H(Z)\otimes A(M_x)\to A(M_x)$. 
\item In particular, when $Z=M\times_{M_0}M$, the $R$-module homomorphism above extends to a homomorphism of $R$-algebras $A_H(M\times_{M_0} M)\to\End_R(A(M_x))$.
\end{enumerate}
\end{prop}
\begin{proof}
We only need to prove this when $A$ is the algebraic cobordism $\Omega$.

Although $p_{12}^{-1}(Z_{12})$ may not be a smooth variety, the projection $p_{12}:p_{12}^{-1}(Z_{12})\to Z_{12}$ is a smooth morphism, as $M_3$ is smooth. Therefore, we have a pull-back $p_{12}^*:\Omega_H(Z_{12})\to \Omega_H(p_{12}^{-1}(Z_{12}))$. Similarly, we have $p_{23}^*:\Omega_H(Z_{23})\to \Omega_H(p_{23}^{-1}(Z_{23}))$. Both $p_{12}^{-1}(Z_{12})$ and $p_{23}^{-1}(Z_{23})$ are $H$-stable subvarieties of $M_1\times M_2\times M_3$, which is smooth, therefore, by \eqref{eq:cap}, have $\cap:\Omega_H(p_{12}^{-1}(Z_{12}))\otimes\Omega_H(p_{23}^{-1}(Z_{23}))\to \Omega_H(Z_{12}\times_{M_2}Z_{23})$. The map $*$ we are looking for is the composition of $p_{12}^*\otimes p_{23}^*$, $\cap$ followed by the push-forward induced by $Z_{12}\times_{M_2}Z_{23}\to Z_{12}\circ Z_{23}$. 

The associativity (2) is proved in the usual way. The statements (3), and (4)are consequences of (2).
\end{proof}


\subsection{Bivariant Riemann-Roch}

In this subsection, we still assume  Assumption~\ref{assum:running} and in addition $R\supset \Q$.

Recall from \S\ref{subsec:HFrational} that $\fl(x)\in R\lbr x\rbr$ is the power series defining the isomorphism between $(R,F)$ and $(R,F_a)$. For each $H$-equivariant vector bundle $V$ over $X$ with Chern roots $\calL_1,...,\calL_n$, define the Todd class
\[\Td(V)=\prod_i \frac{c_1(\calL_i)}{\fl(c_1(\calL_i))}.\]

For any $X$ in $\Sm^H_k$, we say a cohomology class $\eta\in \orH(X)$ is {\em of geometric origin}, if there are $H$-equivariant line bundles $\calL_1,\dots,\calL_k$ on $X$ and a power series $f\in R[\![x_1,\dots,x_k]\!]$ such that $\eta=f(\widetilde{c_1}(\calL_1),\dots,\widetilde{c_1}(\calL_k))(1_X)$. 

Recall the equivariant twisted Chow ring $\ch_\fl$ is defined as follows: for any $X\in \Sm^H_k$, define $\ch_{H,l}(X)=\ch_{H}(X)\hat{\otimes}_\ZZ R$; for any  smooth morphism $f:X\to Y$ in $\Sm^H_k$ the pull-back is $f^*:=f^{\ch *}\hat{\otimes}_\ZZ \id_R:\ch_H(Y)\hat{\otimes}_\ZZ R\to \ch_H(X)\hat{\otimes}_\ZZ R$, and for any proper map $f:X\to Y$ in $\Sm^H_k$ the push-forward $f_*:\ch_H(X)\hat{\otimes}_\ZZ R\to \ch_H(Y)\hat{\otimes}_\ZZ R$ sends $c_1^{\ch}(\calL)$ to $f^{\ch}_*\left(\fl(c_1^{\ch}(\calL))\Td(Tf)\right)$ where  $Tf$ is the relative tangent bundle over $X$. For any $X$, the above isomorphism maps  $c_1(\calL)$ to $\fl(c_1^{\ch}(\calL))$ for any equivariant line bundle $\calL$ on $X$. 

\begin{prop}[\cite{{PaninRR}}] \label{prop:RR}
Let $A$ be an equivariant oriented cohomology theory with formal group law $(R,F)$, satisfying Assumption~\ref{assum:running}, and suppose $\Q\subseteq R$. Then there is an isomorphism of equivariant oriented cohomology theories
\[A(\_)\cong \ch_{\fl}(\_).\]
\end{prop}

 Note that for any $M\in \Sm^H_k$ and any $\eta\in \orH(M\times M)$ of geometric origin, i.e., \[\eta=f(\widetilde{c_1}(\calL_1),\dots,\widetilde{c_1}(\calL_k))(1_{M\times M}), \]
 we denote  \[\fl(\eta)=f\left(\fl(\widetilde c^{\ch}_1(\calL_1)),\dots,\fl(\widetilde c_1^{\ch}(\calL_k))\right)(1_{M\times M})\in \Ch_{H,\fl}(M\times M).\] We define the bivariant Riemann-Roch map $RR$ as   \[RR:A_H(M\times M)\to \ch_{H,\fl}(M\times M), \quad RR(\eta)=\fl(\eta) \pi_2^*(\Td(TM)),\]
where $\pi_2:M\times M\to M$ is the projection to the second factor. Note that the target of $RR$ is the usual equivariant Chow ring extended by scalars and then completed, to be distinguished from the usual twisted (ordinary) Chow ring $\ch_\fl$. From Proposition \ref{prop:RR}, we have the following fact which can be proved the same way as \cite[Proposition~5.11.11]{CG}.
\begin{corollary}\label{cor:RR}
Let $\eta,\eta'\in \orH(M\times M) $ be of geometric origin, then \[RR(\eta* \eta')=RR(\eta)* RR(\eta'),\]
that is, $RR$ is an isomorphism of convolution algebras.
\end{corollary}

\section{Convolution construction of the formal affine Hecke algebra}\label{sec:conv_Hecke}
In this section, we identify the formal affine Hecke algebra as the convolution algebra of the Steinberg variety. 

\subsection{Construction and the main theorem}From now on,  $G$ will be a semi-simple, simply-connected algebraic group with Lie algebra $\lieg$ over an algebraically closed field $k$. Fix $T\subseteq B\subseteq G$ where $T$ a maximal torus and $B$ is a Borel subgroup containing $T$, and let $\La$ be the group of characters of $T$, then there is an associated simply connected root datum $\Sigma\hookrightarrow \La^\vee$ such that $W=N_G(T)/T$ and $\La=\La_w$ is the weight lattice. We identify $G/B$ with the flag variety $\calB$, parametrizing all Borel subalgebras of $\lieg$, then $\La\cong\hbox{Pic}(\calB)$. Let $\calN$ be the nil-cone of $\lieg$ consisting of nilpotent elements. It admits a resolution of singularities given by $\widetilde{\calN}:=T^*\calB$.
The group $G$ acts on $\calB$ naturally and acts on $\widetilde{\calN}$ via the induced action. There is also an action of $\Gm$ on $\widetilde{\calN}$ via scaling each fiber of the cotangent bundle. 

Let $A$ be an arbitrary equivariant  oriented cohomology theory, whose associated formal group law is denoted by $(R,F)$ where $R:=A(k)$ and $F(u,v)\in R[\![u,v]\!]$. 
Let $k_q$ be the standard 1-dimensional representation of $\Gm$, whose $\Gm$-equivariant first Chern class (or weight for short) is denoted by $x_\gamma\in A_{\Gm}(\pt)$.

\begin{prop}
\label{prop:cohpoint}There is an isomorphism of $R$-algebras
\[A_{G\times\Gm}(\widetilde{\calN})\cong A_{G\times\Gm}(\calB)\cong A_{G\times \Gm}(G/T)\cong  A_{T\times\Gm}(k)\cong R[\![x_\gamma]\!][\![\Lambda]\!]_F.\]
\end{prop}
\begin{proof}The first and second isomorphism  follow from extended homotopy equivalence. The third identity follows from the definition of equivariant oriented cohomology, and the last one follows from \cite[Theorem 3.3]{CZZ3}. 
\end{proof}

From now on we denote $S=R[\![x_\gamma]\!][\![\Lambda]\!]_F$. For any element $u\in S$, there is a corresponding element in $\End_R(S)$ acting on $S$ via left multiplication by $u$. 
\begin{lemma}\label{lem:scalar}
Let $\diag:\widetilde{\calN}\to \widetilde{\calN}\times \widetilde{\calN}$ be the diagonal embedding and $\proj_i:\widetilde{\calN}\times \widetilde{\calN}\to \widetilde{\calN}$ be the projections for $i=1,2$. Let $u\in S$. Then the following operator on $S$ \[\proj_{1*}(\diag_*(u)\cdot\proj_2^*(\_))\] is left multiplication by $u$.
\end{lemma}
\begin{proof}For any $v\in S$, we have
\begin{gather*}
\proj_{1*}(\Delta_*(u)\cdot \proj^*_2(v))=\proj_{1*}(\Delta_*(u\cdot \Delta^*(\proj^*_2(v))))=u\cdot \Delta^*(\proj_2^*(v))=uv
\end{gather*}
where the first identity follows from the projection formula, and the other two identities follows from the identities $\proj_i\circ \Delta=\id_{\tilde N}$ for $i=1,2$.
\end{proof}

Recall that in $\calB\times\calB$, the orbits of the diagonal $G$-action are in natural one-to-one correspondence with the Weyl group $W$. Let $\calY_\al$ be the orbit corresponding to a simple root $\al\in \Phi$. Its closure $\overline{\calY_\alpha}$ is the union of $\calY_\alpha$ and $\calB_\diag$, the diagonal. Let $Z_\alpha$ be $T^*_{\overline{\calY_\alpha}}(\calB\times\calB)$, which is the closure of $T^*_{\calY_\al}(\calB\times \calB)$, considered as a closed subvariety of $\widetilde{\calN}\times_{\calN}\widetilde{\calN}$. Note that $Z_\alpha$ is smooth for any simple root $\alpha$. 
Let $\pi:Z_\alpha\to \overline{\calY_\alpha}$ be the bundle projection. Note that $Z_\alpha$ is smooth, and the second projection $\proj_2:\widetilde{\calN}\times\widetilde{\calN}\to \widetilde{\calN}$ is proper when restricted to $Z_\alpha$. Without raising too much confusion, we will denote the restriction of $\proj_i$ to $Z_\alpha$ still by $\proj_i$ for $i=1,2$. 
According to Example~\ref{ex:conv_sm}, any element $\eta\in A_{G\times\Gm}(Z_\alpha)$ defines an operator in $\End_R(A_{G\times\Gm}(\widetilde{\calN}))$ by sending any $\mu\in A_{G\times\Gm}(\widetilde{\calN})$ to $\proj_{1*}(\proj_2^*(\mu)\cdot \eta)\in A_{G\times\Gm}(\widetilde{\calN})$. This operator will be denoted by $\eta*_Z$.

From now on in this section, we denote $c_1=c_1^{G\times \Gm}$ for simplicity. Let $\pi_2:\overline{\calY_\al}\to \calB$ be the second projection $\calB\times\calB\to \calB$ restricted to $\overline{\calY_\al}$, which is a $G\times \Gm$-equivariant fibre bundle with each fiber isomorphic to $\PP^1$. Let $\Omega^1_{\pi_2}$ be the relative cotangent bundle of the projection ${\pi_2}:\calB\times\calB\to \calB$, which  is a $G\times \Gm$-equivariant line bundle. Define $\calJ_\alpha=\pi^*\Omega^1_{\pi_2}$ on $Z_\alpha$, and 
\begin{equation}\label{eqn:Q_alpha}
J^A_\alpha:=\frac{c_1(\calJ_\alpha)-c_1( k_q)}{c_1(\calJ_\alpha\otimes k_q^\vee)},
\end{equation} 
which belongs to  $A_{G\times\Gm}(Z_\alpha)$ by Lemma~\ref{lem:fgl1} .

Let $A_{G\times\Gm}(Z)'$ be the $R$-subalgebra of $\End_R(A_{G\times\Gm}(\widetilde{\calN}))$ generated by $u\in A_{G\times \Gm}(\widetilde \calN)$ and  operators $J^A_\alpha*_Z$ for $\alpha\in\Phi$. 
Note that by Lemma~\ref{lem:scalar}, for $u\in A_{G\times \Gm}(\widetilde \calN)\cong S$, the corresponding operator in $ \End_R(A_{G\times\Gm}(\widetilde{\calN}))$ acts by left multiplication by $u$. This is different from the convention in $K$-theory where $x_\lambda\in S$ acts by multiplication by $x_{-\lambda}$ \cite[Proposition 7.6.38]{CG}. Here $x_\lambda=c_1(\calL_\lambda)$ with $\calL_\lambda$ the line bundle on $\calB$ with character $\la\in \Lambda$. 

Let  $\mu_\al:=\frac{x_{-\al}}{-x_\al}.$ Recall from Remark \ref{rmk:gen_J} that we have $J^F_\al\in \HF$ with 
\begin{equation}\label{eq:JFal}
J_{\al}^F\cdot x_\la=(x_{-\al}-x_\ga)Y_{-\al}\cdot x_\la=(x_{-\al}-x_\ga)(\frac{x_\la}{x_{\al}}+\frac{s_\al(x_\la)}{x_{-\al}}), ~\la\in \La.
\end{equation}

\begin{theorem}\label{thm:main}Under  Assumption \ref{assump:strong},  there is an isomorphism of $R$-algebras \[\Psi_A:\HF\to A_{G\times\Gm}(Z)'\]
 defined by $\Psi_A(x)=x\in S\cong A_{G\times\Gm}(\widetilde{\calN})$ for $x\in S\subseteq \HF$ and  $\Psi_A(J^F_\al)=J^A_{\al}$ such that the following diagram commutes
 \begin{equation}
 \label{eq:maindiag}\xymatrix{
\HF\ar[r]^-{\Psi_A}\ar[d]&A_{G\times\Gm}(Z)'\ar[d] \\
\End_R(S)\ar[r]^-\cong & \End_R(A_{G\times\Gm}(\widetilde{\calN})).
}
\end{equation}
\end{theorem}

By definition the right vertical map in the diagram \eqref{eq:maindiag} is injective, and by Corollary \ref{cor:faithact} the left vertical map is also injective. The action of $x\in A_{G\times \Gm}(Z)'$ and $\Psi_A(x)\in \HF$ coincide by definition. To prove the theorem, it suffices to show that the action of $J_\al^A$ on $x_\la=c_1(\calL_\la)$ coincides with \eqref{eq:JFal}. 

Following the same reduction argument as in \cite[\S7.6]{CG}, it suffices to assume that  $G$ is of rank-one. We will calculate in \S\ref{subsec:SL2}  the effect of the operator $J^A_\alpha$  on $c_1(\calL_\lambda)$, which is  given by Proposition~\ref{prop:Q_alpha} below.  The proof of Theorem \ref{thm:main} then follows.

We need the following lemma to simplify our calculation. For each $\eta\in A_{G\times \Gm}(\calB\times \calB)$, we denote the convolution operator $\eta*_{\calB\times\calB}\in \End(A_{G\times \Gm}(\calB))$ by $\eta*_{\calB\times\calB}\mu=\pi_{1*}(\eta\cdot\pi_2^*(\mu))$ where $\pi_i$ are the projections $\calB\times \calB\to \calB$ for $i=1,2.$
\begin{lemma} \cite[Lemma~5.4.27]{CG}\label{Lem:convol_zerosec}
Let $j:Z_\al\to \widetilde{\calN}\times\widetilde{\calN}$ be the natural embedding. Let $p_2:\widetilde{\calN}\times\widetilde{\calN}\to \widetilde{\calN}\times\calB$ be the identity on the first factor and the bundle projection on the second  factor. Let $i:\calB\times\calB\to \widetilde{\calN}\times\calB$ be the zero-section. Then, $p_2\circ j$ is injective. Moreover, the following diagram commutes 
\[\xymatrix{A_{G\times\Gm}(Z_\al)\ar[r]^-{*_Z}\ar[d]^-{i^*\circ p_{2*}\circ j_*}& \End_R(S)\ar[d]^-{\cong}\\
A_{G\times\Gm}(\calB\times\calB)\ar[r]^-{*_B}&\End_R(A_{G\times\Gm}(\calB)).
}\]
\end{lemma}

\subsection{Rank-one case}\label{subsec:SL2}
In this subsection we  assume $G$ has rank 1. There is only one simple root, denoted by $\alpha$.
In this case, $\calB\cong\PP^1$, and $Z_\alpha=\overline{\calY_\alpha}=\PP^1\times\PP^1$. The line bundle $\calL_\lambda$ is isomorphic to $\calO_{\PP^1}(\angl{\lambda,\alpha^\vee})$ on $\PP^1$, and the line bundle $\Omega^1_{\pi_2}$ is identified with $\Omega^1_{\PP^1}\boxtimes \calO_{\PP^1}:=\pi_1^*(\Omega^1_{\PP^1})\otimes \pi_2^*\calO_{\PP^1}$. The composition of $j$ and $p_2$ is the zero-section of the line bundle $\Omega^1_{\PP^1}\boxtimes \calO_{\PP^1}\cong (\Omega^1_{\PP^1}\otimes k_q^\vee)\boxtimes \calO_{\PP^1}$, and so is $i$. The map 
\begin{equation}
\label{eq:Thom}i^*\circ p_{2*}\circ j_*:A_{G\times\Gm}(Z_\al)=A_{G\times \Gm}(\PP^1\times \PP^1)\to A_{G\times\Gm}(\PP^1\times \PP^1),
\end{equation} according to  \textbf{A} 8, 
 (see also \cite[(7.5.17)]{CG} for the $K$-theory version),  
is multiplication by $c_1((\Omega^1_{\PP^1}\otimes k_q^\vee)\boxtimes \calO_{\PP^1})$. 

We identify $\PP^1$ with $\PP(\Aff^2)$, and $T$ with $\Gm$, then the action of $T$ on $\Aff^2$ has weights $\alpha/2$ and $-\alpha/2$.

\begin{lemma}\label{lem:cohw/ring}
Under the identification $A_{G\times\G_m}(\PP^1)\cong A_{T\times\G_m}(k)\cong R[\![x_\gamma]\!][\![x_{\alpha/2}]\!]$ from Proposition \ref{prop:cohpoint}, the class $c_1(\calO_{\PP^1}(1))$ is identified with $x_{\alpha/2}$.
\end{lemma}
\begin{proof} This follows from the proof of \cite[Theorem 3.3]{CZZ3}.
\end{proof}


Let $\pi_i:\PP^1\times\PP^1\to \PP^1$ be the $i$-th projection for $i=1,2$.

\begin{lemma} \label{lem:pushpullrank1}
The element $1\in A_{G\times\G_m}(\PP^1\times\PP^1)$ acts on $A_{G\times \Gm}(\PP^1)\cong R\lbr x_\ga, x_{\alpha/2}\rbr$ by   \[1*_{\calB\times\calB} c_1(\calL_\lambda)=\frac{\angl{\lambda,\alpha^\vee}_Fx_{-\alpha/2}}{x_{-\alpha/2}-_Fx_{\alpha/2}}+\frac{\angl{\lambda,\alpha^\vee}_Fx_{\alpha/2}}{x_{\alpha/2}-_Fx_{-\alpha/2}}\]
where $\calL_\la$ is the line bundle on $\PP^1$ with character $\la\in \La$.
\end{lemma}

\begin{proof}We have
\[
\pi_{1*}\left(1\cdot \pi_2^*(c_1\calL_\lambda)\right))
=\pi_{1*}\left(\pi_2^*(c_1\calO_{\PP^1}(\angl{\lambda,\alpha^\vee}))\right)
\overset{\eqref{Quill-formula}}=\frac{\angl{\lambda,\alpha^\vee}_Fx_{-\alpha/2}}{x_{-\alpha/2}-_Fx_{\alpha/2}}+\frac{\angl{\lambda,\alpha^\vee}_Fx_{\alpha/2}}{x_{\alpha/2}-_Fx_{-\alpha/2}}.
\]
\end{proof}

By the projection formula and Lemma \ref{lem:pushpullrank1}, we have the following more general formula.

\begin{lemma}\label{lem:conv}
Let $f(x_\alpha,x_\gamma)\in R[\![x_{\alpha/2}]\!][\![x_\gamma]\!]\cong A_{G\times\G_m}(\PP^1)$ be an arbitrary class. Then the element $\pi_1^*(f)\in A_{G\times\G_m}(\PP^1\times\PP^1)$ acts on $c_1(\calL_\lambda)\in A_{G\times \Gm}(\PP^1)$ by  \[(\pi_1^*f)*_{\calB\times\calB}c_1(\calL_\lambda)=f(x_\alpha,x_\gamma)\left(\frac{\angl{\lambda,\alpha^\vee}_Fx_{-\alpha/2}}{x_{-\alpha/2}-_Fx_{\alpha/2}}+\frac{\angl{\lambda,\alpha^\vee}_Fx_{\alpha/2}}{x_{\alpha/2}-_Fx_{-\alpha/2}}\right).\]
\end{lemma}

\begin{prop}\label{prop:Q_alpha}
We have $J^A_\alpha\in A_{G\times\G_m}(Z_\al)$ is equal to the class 
\[\frac{c_1(\calO(-2))-c_1( k_q)}{c_1(\calO(-2)\otimes k_q^\vee)}.\] Moreover, the effect of the operator $J^A_\al*_Z\in A_{G\times \Gm}(Z)'$ on $c_1(\calL_\la)\in A_{G\times \Gm}(\PP^1)$ coincides with that of $J^F_{\al}$ on $x_\la\in S$.
\end{prop}

\begin{proof}
The first part follows from the fact that $\Omega^1_{\PP^1}=\calO_{\PP^1}(-2)$. For the second part, we have
\begin{eqnarray*}
J^A_\alpha*_Z c_1(\calL_\lambda)&=&\pi_{1*}\left[\pi_2^*\left(c_1(\calL_\lambda)\right)\cdot c_1\left((\Omega^1_{\PP^1} \otimes k_q^\vee)\boxtimes\calO_{\PP^1}\right)\cdot J^A_\alpha\right]\\
&=&\pi_{1*}\left[\pi_2^*\left(c_1(\calO_{\PP^1}(\angl{\lambda,\alpha^\vee}))\right)\cdot \pi_1^*\left(c_1(\calO_{\PP^1}(-2))-c_1( k_q)\right)\right]\\
&=&(x_{-\alpha}-x_{\gamma})\left(\frac{(-\angl{\lambda,\alpha^\vee})_Fx_{\alpha/2}}{x_{-\alpha/2}-_Fx_{\alpha/2}}+\frac{(-\angl{\lambda,\alpha^\vee})_Fx_{-\alpha/2}}{x_{\alpha/2}-_Fx_{-\alpha/2}}\right)\\
&=&(x_{-\alpha}-x_{\gamma})\cdot\left(\frac{s_\alpha\left(x_\lambda\right)}{x_{-\alpha}}+\frac{x_\lambda}{x_\alpha}\right)
\end{eqnarray*}
where the first identity follows from Lemma~\ref{Lem:convol_zerosec}, \eqref{eq:Thom} and the definition of $*_B$, and the second identity follows from the definition of $J_\alpha^A$. The third identity follows from Lemma \ref{lem:conv} and the projection formula. The fourth identity follows since the rank is assumed to be one. Comparing with \eqref{eq:JFal} we know that the effect of $J^A_\al$ on $c_1(\calL_\la)$ coincides with that of $J^F_{\al}$ on $x_\la$, so the conclusion follows.
\end{proof}

\subsection{Bivariant Riemann-Roch} \label{subsec:riemroch} In this subsection assume that $R$ is a $\QQ$-algebra.
As is proved in \S\ref{subsec:HFrational},  all the formal affine Hecke algebras are isomorphic to each other regardless of the formal group laws. 
In the following Proposition, we prove that the bivariant Riemann-Roch map defined in Corollary~\ref{cor:RR}  provides  isomorphisms between convolution algebras and is compatible with the one in \S\ref{subsec:HFrational}. Recall from \S\ref{subsec:fga} that $\fl(x)$ is the power series such that $\fl(x+y)=\fl(x)+_F\fl(y)$. 

\begin{prop}\label{prop:RR_geom_alg}
Under Assumptions~\ref{assum:running}, \ref{assump:strong}, and $\bbQ\subseteq R$. 
Let $RR:A_{G\times\Gm}(\calB\times\calB)\to \ch_{G\times\Gm}(\calB\times\calB;\QQ)\hat{\otimes}_\QQ R$ be the bivariant Riemann-Roch map as in Corollary~\ref{cor:RR}. Then $RR$ induces an isomorphism $RR:A_{G\times\Gm}(Z)'\to\ch_{G\times\Gm}(Z;\QQ)'\hat{\otimes} R$ making the following diagram commutative, \[\xymatrix{
A_{G\times\Gm}(Z)'\ar[r]^-{RR}\ar[d] &\ch_{G\times\Gm}(Z;\QQ)'\hat{\otimes} R\ar[d]\\
A_{G\times\Gm}(\calB\times\calB)\ar[r]^-{RR}& \ch_{G\times\Gm}(\calB\times\calB;\QQ)\hat{\otimes}_\QQ R
}\] where the vertical maps are given by Lemma~\ref{Lem:convol_zerosec}.

Moreover, the following diagram of isomorphisms commutes:
\[\xymatrix{
\HF\ar[r]^-{\widetilde\phi_{F,F_a}}\ar[d]_\sim^{\Psi_A}&\HF^a\ar[d]_\sim^{\Psi_{\ch}}\\
A_{G\times\Gm}(Z)'\ar[r]^-{RR} &\ch_{G\times\Gm}(Z;\QQ)'\hat{\otimes} R
}\]
where the map $\widetilde \phi_{F,F_a}$ is defined in \S\ref{subsec:HFrational}.
\end{prop}
\begin{proof}It suffices to check the commutativity on $J^A_\al$ for simple root $\alpha$, and  it suffices to prove this in the case when $G=\SL_2$, in which case $Z_\alpha=\PP^1\times\PP^1$ and the embedding $Z_\alpha\inj \widetilde{\calN}\times\widetilde{\calN}\cong T^*(\PP^1\times \PP^1)$ is identified with the zero-section of $T^*(\PP^1\times\PP^1)$. 
By Lemma~\ref{Lem:convol_zerosec}, $i^*p_{2*}j_*(J^A_\alpha)=\pi_1^*(c_1(\calJ_\alpha)-c_1(k_q)) \in A^*_{G\times\Gm}(\PP^1\times\PP^1)$.  Note that in the Chow ring,  $J_\alpha^{\ch}=c_1^{\ch}(\pi_2^*\calJ_\alpha)$. Let $\pi_i:\PP^1\times\PP^1\to \PP^1$ be the $i$-th projection for $i=1,2$.
We have 
\begin{eqnarray*}
RR(J_\alpha^A)&=&\fl\left(\pi_1^*(c_1^{\ch}(\calJ_\alpha)-c_1^{\ch}(k_q))\right)
\cdot \left(\frac{c_1^{\ch}\pi_2^*T\PP^1}{\fl( c_1^{\ch}\pi_2^*T\PP^1)}\right)\\
&=&\left(\frac{\fl( c_1^{\ch}\pi_1^*T^*\PP^1)}{c_1^{\ch}\pi_1^*T^*\PP^1}\right)\cdot\left(\frac{c_1^{\ch}\pi_2^*T\PP^1}{\fl( c_1^{\ch}\pi_2^*T\PP^1)}\right)J_\alpha^{\ch}.
\end{eqnarray*}
Since $\frac{\fl(x)}{x}$ is an invertible power series, so  the map $RR$ is an isomorphism. By Lemma~\ref{lem:conv}, the element \[RR(J_\al^A)=\left(\frac{\fl( c_1^{\ch}\pi_1^*T^*\PP^1)}{c_1^{\ch}\pi_1^*T^*\PP^1}\right)\cdot\left(\frac{c_1^{\ch}\pi_2^*T\PP^1}{\fl( c_1^{\ch}\pi_2^*T\PP^1)}\right)J_\alpha^{\ch}
\]
 is the operator
\[x_\lambda\mapsto \left(\fl(x_{-\alpha})-\fl(x_\gamma)\right)C^{F_a}_{-\alpha}(\frac{x_\alpha}{\fl (x_\alpha)}\cdot x_\lambda).\]
For $F=F_a$ we have $\ka^{F_a}=0$ and $x_{-\al}=-x_\al$, so the right hand side is equal to the action of the operator $(\fl(x_{-\al})-\fl(x_\ga))(-X^{F_a}_{-\al})\frac{x_\al}{\fl(x_\al)}$ on $x_\la$.  Denote $\fl_\al=\fl(x_\al), \fl_{-\al}=\fl(-x_\al)$ and $\fl_\ga=\fl(x_\ga)$, by Theorem \ref{thm:main} we have
\begin{gather*}
\Psi_{\ch}^{-1}\circ RR\circ \Psi_A(J^F_{\al})=RR(J^A_\al)=(\fl_{-\al}-\fl_\ga)(-X^{F_a}_{-\al})\frac{x_\al}{\fl_\al}=(\fl_{-\al}-\fl_\ga)[\frac{1}{x_{-\al}}\de_\al-\frac{1}{x_{-\al}}]\frac{x_\al}{\fl_\al}\\=(\fl_{-\al}-\fl_\ga)[\frac{1}{-x_{\al}}\frac{x_{-\al}}{\fl_{-\al}}\de_\al+\frac{1}{\fl_\al}]=(\fl_{-\al}-\fl_\ga)(\frac{1}{\fl_{-\al}}\de_\al+\frac{1}{\fl_{\al}}).
\end{gather*}
From \S\ref{subsec:HFrational} we have
\begin{gather*}\widetilde \phi_{F,F_a}(J^F_{\al})=\widetilde \phi_{F,F_a}((x_{-\al}-x_\ga)Y_{-\al})=\widetilde\phi_{F,F_a}\left((x_{-\al}-x_\ga)(\frac{1}{x_{-\al}}\de_\al+\frac{1}{x_{\al}})\right)
=(\fl_{-\al}-\fl_\ga)(\frac{1}{\fl_{-\al}}\de_\al+\frac{1}{\fl_{\al}}),
\end{gather*}
so the diagram commutes.
\end{proof}

\subsection{Equivariant oriented cohomology of the Steinberg variety}

In this subsection we assume the  oriented cohomology theory $A_H$ satisfies  Assumption~\ref{assum:running}, and the formal group law $(R,F)$ satisfies Assumption~\ref{assump:strong}. We prove that the algebra $A_{G\times\Gm}(Z)'$ is actually isomorphic to $A_{G\times\Gm}(Z)$. Note that this statement has no meaning without Assumption~\ref{assum:running} and the interpretation in \S\ref{subsec:coh_cob}. Without loss of generality, we assume $A_H=\Omega_H$ and  so $F=F_{\Laz}$. 

 Following the same argument as in \cite[Claim~7.6.7]{CG} (which only relies on formal properties of equivariant oriented cohomology theories), the representation $\Omega_{G\times\Gm}(Z)\to \End_R(\Omega_{G\times\Gm}(\tilde{\calN}))$ is faithful.
 
\begin{prop}
The assignment $J^{\Laz}_{\al}\mapsto J^\Omega_{\al}\in \Omega_{G\times \Gm}(Z_\al)\subset \Omega_{G\times \Gm}(Z)$ for any simple root $\alpha$ extends to a unique injective ring homomorphism \[
\Theta:\bfH_{\Laz}\to \Omega_{G\times \Gm}(Z).\]
\end{prop}
\begin{proof}
The class $J_\alpha^\Omega\in \Omega_{G\times \Gm}(Z)$ acts on $S$ by the operator $J_\alpha^{\Laz}$. Hence, the image of $\Omega_{G\times \Gm}(Z)$ in $\End_R(S)$ contains the operators $J_\alpha^{\Laz}$. The ring homomorphisms $\bfH_{\Laz}\cong\Omega_{G\times \Gm}(Z)' \to \End_R(\Omega_{G\times\Gm}(\tilde{\calN}))$ and $\Omega_{G\times \Gm}(Z)\to \End_R(\Omega_{G\times\Gm}(\tilde{\calN}))$ are injective. Therefore, the statement follows from the fact that $\bfH_{\Laz}$ is generated by $J_\alpha^{\Laz}$ and $S$.
\end{proof}

When restricting to special equivariant cohomology theories, Theorem~\ref{thm:main} has the following strengthening.
\begin{theorem}\label{thm:isom_HF_cob} Under Assumption \ref{assump:strong} and \ref{assum:running}, the morphism $\Theta$ is an isomorphism of algebras.
\end{theorem}

This theorem is proved in the same way as \cite[Theorem~7.6.10]{CG}. For completeness, we include the sketch of the proof.


On $\bfH_{\Laz}$, there is a well-defined filtration by Bruhat order. Explicitly, for any set of reduced sequences $\{I_w\}_{w\in W}$, let $\bfH_{\Laz}^{\leq w}\subseteq\bfH_{\Laz}$ be the $S$-submodule spanned by $J^{\Laz}_{I_v}$ with $v\le w$. This filtration is well-defined since $\{J^{\Laz}_{I_w}\}_{w\in W}$ is a $S$-basis of $\bfH^{\Laz}$, according to Remark \ref{rem:JFproperty}. Indeed, this filtration does not depend on the choice of $\{I_w\}_{w\in W}$, and replacing $J^{\Laz}$ by $T^{\Laz}$ will give the same filtration.  The subquotient $\bfH_{\Laz}^{\leq w}/\bfH_{\Laz}^{<w}$ is free of rank one, spanned by $J_{I_w}^{\Laz}$.

On the other hand, we define a filtration on $\Omega_{G\times\Gm}(Z)$. Let $\calY_w$ be the $G$-orbit of $\calB\times \calB$ corresponding to $w\in W$. By \cite[Corollary 3.3.5]{CG}, we have $Z=\coprod_{w\in W}T^*_{\calY_w}(\calB\times \calB)$. The irreducible components of $Z$ are in one-to-one correspondence with elements in the Weyl group $W$. For example the closure $\overline{T^*_{\calY_{s_i}}(\calB\times \calB)}=T^*_{\overline{\calY_{s_i}}}(\calB\times \calB)$ is the irreducible component corresponding to the simple reflection $s_i$. For any $w\in W$, let 
\[Z_{\leq w}=\coprod_{v\leq w}T^*_{\calY_v}(\calB\times \calB).
\]
 This is a closed subvariety of $Z$, stable under the action of  $G\times\Gm$. The push-forward morphisms $\Omega_{G\times\Gm}(Z_{\leq w})\to \Omega_{G\times\Gm}(Z)$ are injective whose images define a filtration on $\Omega_{G\times\Gm}(Z)$ by Bruhat order. As $Z_{\leq w}=Z_{<w}\coprod T^*_{\calY_w}(\calB\times \calB)$, and $T^*_{\calY_w}(\calB\times \calB)$ is an affine bundle over $\calY_w$, which in turn is an affine bundle over $\calB$, we have 
\begin{equation}\label{eq:Zfil}\Omega_{G\times\Gm}(Z_{\leq w})/\Omega_{G\times\Gm}(Z_{< w})\cong \Omega_{G\times\Gm} (T^*_{\calY_w}(\calB\times \calB))\cong \Omega_{G\times \Gm}(\calB)\cong  S.
\end{equation}

\begin{proof}[Proof of Theorem \ref{thm:isom_HF_cob}]
By definition of $\Theta$ and the proof of  \cite[Proposition 7.6.12.(1)]{CG}, the map $\Theta$ is filtration preserving, so by \cite[Proposition 2.3.20.(ii)]{CG} we only need to show that $\Theta$ induces an isomorphism  of $S$-modules
\[\Theta_w:\bfH_{\Laz}^{\leq w}/\bfH_{\Laz}^{<w}\to  \Omega_{G\times\Gm}(Z_{\leq w})/\Omega_{G\times\Gm}(Z_{< w})\cong S.
\]
 Since the left hand side is a free $S$-module with basis $J^{\Laz}_{I_w}$, it suffices to show that  $\Theta_w(J^{\Laz}_{I_w})$ is invertible in  $S$.

Let $I_w=(i_1,...,i_r)$. We need to calculate $\Theta_w(J^{\Laz}_{I_w})$ which is equal to   the convolution by $J^\Omega_{s_{i_1}}*\cdots *J^\Omega_{s_{i_r}}$.  Following the proof of \cite[Proposition 7.6.12.(2)]{CG}, we define $\proj_{j,j+1}:(T^*\calB)^{r+1}\to T^*(\calB\times\calB)\cong T^*(\calB)\times T^*(\calB)$ to be the projection to the $(j,j+1)$ factor, and define $\calZ_{i_j}=\proj_{j,j+1}^{-1}(T^*_{Y_{s_{i_j}}}(\calB\times\calB))$, for $j=1,\dots,r-1$.  The projection
\[\calZ_{{i_1}}\times_{\tilde{N}} \calZ_{{i_2}}\times_{\tilde{N}}\cdots\times_{\tilde{N}}\calZ_{{i_r}}\to Z_w\] restricts to an isomorphism \[\calZ_{i_1}\cap\calZ_{i_2}\cap\cdots\cap\calZ_{i_r}\cong T^*_{\calY_w}(\calB\times\calB).\] Therefore, we have
 \begin{eqnarray*}
\Theta_w(J^{\Laz}_{I_w})&=&\Theta(J^{\Laz}_{s_1}\cdots J^{\Laz}_{s_r})|_{T^*_{\calY_w}(\calB\times \calB)}\\
&=&(J^\Omega_{s_{i_1}}*\cdots *J^{\Omega}_{s_{i_r}})|_{T^*_{\calY_w}(\calB\times \calB)}\\
&=&(\frac{c_1(\calJ_{i_1})-c_1( k_q)}{c_1(\calJ_{i_1}\otimes k_q^\vee)}|_{\calZ_1})\cap\cdots\cap(\frac{c_1(\calJ_{i_r})-c_1( k_q)}{c_1(\calJ_{i_r}\otimes k_q^\vee)}|_{\calZ_r})\\
&=& \frac{c_1(\calJ_{i_1})-c_1( k_q)}{c_1(\calJ_{i_1}\otimes k_q^\vee)}\cdots\frac{c_1(\calJ_{i_r})-c_1( k_q)}{c_1(\calJ_{i_r}\otimes k_q^\vee)}|_{T^*_{\calY_w}(\calB\times \calB)},
\end{eqnarray*}
which is invertible in $S$ by Lemma~\ref{lem:fgl1}.
\end{proof}

\section{Geometric construction of the standard modules}
In this section, we study representations of the algebra $\HF$. In particular, we identify the set of all irreducible representations, and use Theorem~\ref{thm:main} to show that for any $x\in \calN$, we have an action of $\HF$ on $A(\calB_x)$.

Let $k_\chi$ be a field, and let $R\to k_\chi$ be any point  of $\Spec R$. We consider modules  $M$ over $\HF$, such that the action of $R\subseteq \HF$ factors through $R\to k_\chi$. We assume further that $M$ is finite dimensional as a vector space over $k_\chi$. The abelian category of $\HF$-modules with this property will be denoted by $\HF\hbox{-mod}_\chi$. We will study irreducible objects in this category, and give geometric construction of the standard objects. We say $M\in \HF\hbox{-mod}_\chi$ is irreducible if it has no sub-objects. Recall that we defined a filtration on $\HF$ in \S\ref{subsec:PBW}.

\begin{lemma}\label{lem:repn_factor}
For every irreducible representation $M$ in $\HF\hbox{-mod}_\chi$, the action of $\HF$ factors through the quotient algebra \[\HF^{(0)}\cong \HF^0/\HF^1.\]
\end{lemma}

We introduce some notations. The subring $S^W\subseteq \HF$ is a central subalgebra. The action of $\HF$, when restricted to $S^W$, factors through $S^W\to k_\chi$. Let $\IF$ be the augmentation ideal of $S$ as in \S\ref{subsec:PBW}, and denote the intersection $S^W_+=\IF\cap S^W$, then it is contained in the kernel of  $S^W\to k_\chi$. So $S^W_+$ acts trivially on $M$. 
Let $(S^W_+)$ be the ideal in $\HF$ generated by $S^W_+$. We know that $x_\ga\in S^W_+$, so by Example~\ref{ex:HFdeg0}, we have \[\HF/(S^W_+)\cong (S/S^W_+ )\rtimes R[W].\]
Recall also that $\HF$ is of finite rank as a module over $S^W$. Consequently $\HF/(S^W_+)$ is an $R$-module of finite rank.

\begin{proof}[Proof of Lemma~\ref{lem:repn_factor}]
The filtration on $\HF$ in \S\ref{subsec:PBW} induces a filtration on $\HF/(S^W_+)$. Since $\HF/(S^W_+)$ has finite rank, so $\HF^i/(S^W_+)=0$ for some $i$, and therefore $\HF^1/(S^W_+)$ is a nilpotent subset in $\HF/(S^W_+)$. So $\HF^1/(S^W_+)$ is contained in the  Jacobson radical of $\HF/(S^W_+)$, which acts trivially on $M$.
\end{proof}

\begin{remark}
Assuming that the torsion index $\ttt$ is invertible in $R$ (see \cite[\S5.1]{CPZ}), by \cite[Theorem 6.9 and Theorem 13.13]{CPZ},  the ring of coinvariants $S/S^W_+$ of the formal group algebra is isomorphic to $A(\calB)$, which is a free $R$-module of finite rank. Hence in this case \[\HF/(S^W_+)\cong (S/S^W_+ )\rtimes R[W]\cong  A(\calB)\rtimes R[W].\]
\end{remark}

Recall from \S\ref{subsec:PBW} that $\HF^{(0)}$ is canonically isomorphic to $H_{deg}^{(0)}\cong R[W]$. Therefore, we have the following easy consequence.

\begin{corollary}\label{cor:deg_hack_rep}
The irreducible representations of $\HF$  on which $R$ acts by $k_\chi$ are in one-to-one correspondence with those of $k_\chi[W]$.
\end{corollary}

Now we assume the oriented cohomology theory $A$ satisfies  Assumption~\ref{assum:running}, so that we can consider convolution algebra and its representations.
For any $x\in \calN$, let $\calB_x$ be the fiber of the Springer resolution $\widetilde{\calN}\to\calN$ over $x$. Let $G_x$ be the centralizer of $x$ with respect to the $G$-action on $\calN$. Let $C_x$ be the component group of $G_x$, i.e., the quotient of $G_x$ by its connected component containing the identity.

\begin{prop}\label{prop:action_springer_fiber}
Under Assumptions~\ref{assump:strong} and  Assumption~\ref{assum:running}, we have 
\begin{enumerate}
\item There is a natural action of $\HF$ on $A(\calB_x)$, and this action  factors through the quotient algebra $\HF/(S^W_+)$;
\item The action of $\HF$  on $A(\calB_x)$ commutes with the action of $C_x$.
\end{enumerate}
\end{prop}
\begin{proof}
It is a direct consequence of Proposition~\ref{prop:conv_action} that $\HF$ acts on $A(\calB_x)$, and the second part of Statement (1) follows from the observation that the action of $\HF\cong A_{G\times\Gm}(Z)$ factors through $A(Z)$. In other words, the subring $S^W\cong A_{G\times\Gm}(\pt)$ acts by the quotient $S^W\to R$.

Statement (2) follows from the same proof as that of \cite[Lemma~3.5.2]{CG}.
\end{proof}

\begin{remark}
It is proved by De Concini, Lusztig, and Procesi that the Springer fiber $\calB_x$ admits a cell decomposition. Therefore, the rank of $A(\calB_x)$ as a module over $R$ is finite and is independent of the equivariant cohomology theory $A$.
\end{remark}

If $k_\chi$ has characteristic zero, i.e., $\bbQ\subseteq k_\chi$, then by the bivariant Riemann-Roch theorem, the  representations of $\HF$ are no much different then that of degenerate affine Hecke algebras. More precisely, by Corollary~\ref{cor:RR} and Proposition~\ref{prop:RR_geom_alg}, there is an isomorphism $RR: A(\calB_x)\cong\ch(\calB_x)\otimes_{\ZZ}R$, so that the following diagram commutes 
\[\xymatrix{
\HF\ar[r]\ar[d]&\bfH_{F_a}\ar[d]\\
\End_R(A(\calB_x))\ar[r]&\End_R(\ch(\calB_x)\otimes_{\ZZ}R);}\]

For each irreducible representation $\nu$ of $C_x$ over the field $k_\chi$, let $A_{x,\chi,\nu}$ be the isotypical component in \mbox{$A_{G\times\Gm}(\calB_x)\otimes_Rk_\chi$} that transforms under $C_x$ as the representation $\nu$.

\begin{corollary}\label{thm:Deligne-Langlands}
Assume further that $k_\chi$ has characteristic zero. Then 
\begin{enumerate}
\item for any pair $(x,\nu)$, the module $A_{x,\chi,\nu}$ has a unique irreducible quotient if it is non-zero, which will be denoted by $L_{x,\chi,\nu}$;
\item any irreducible $\HF$ -module finitely generated over $R$ such that the $R$-action factors through  $k_\chi$ is isomorphic to one of them;
\item the characters of $A_{x,\chi,\nu}$ are given by the  Deligne-Langlands-Lusztig character formula.
\end{enumerate}   
\end{corollary}

But if $k_\chi$ has positive characteristic, although the absolutely irreducible representations are the same as those of the $k_\chi[W]$ as Corollary~\ref{cor:deg_hack_rep} shows, the study of irreducible objects in $k_\chi[W]\hbox{-mod}$, and hence in $\HF\hbox{-mod}_\chi$ in positive characteristic is difficult to due to the failure of the decomposition theorem of perverse sheaves. Moreover, the Jordan-H\"older multiplicities of the irreducible representations in $A(\calB_x)$ for $x\in \calN$ are not known to us.

\newcommand{\arxiv}[1]
{\texttt{\href{http://arxiv.org/abs/#1}{arXiv:#1}}}
\newcommand{\doi}[1]
{\texttt{\href{http://dx.doi.org/#1}{doi:#1}}}
\renewcommand{\MR}[1]
{\href{http://www.ams.org/mathscinet-getitem?mr=#1}{MR#1}}


\begin{thebibliography}{00}

\bibitem[And03]{And03} M. Ando, {\em The Sigma orientation for circle-equivariant elliptic cohomology},  Geom. Topol., {\bf 7}, (2003), 91--153. \MR{1988282}

\bibitem[BE90]{BE} P. Bressler and  S. Evens, 
{\it The Schubert calculus, braid relations, and generalized cohomology},
Trans. Amer. Math. Soc. 317 (1990), no. 2, 799--811.  \MR{0968883}

\bibitem[Ch10]{Chen} H.-Y. Chen, {\em Torus equivariant elliptic cohomology and sigma orientation}. Ph.D. Thesis, University of Illinois at Urbana-Champaign, 109 pp, (2010). \MR{2873496}

\bibitem[CPZ13]{CPZ}
B. Calm\`es, V. Petrov, and K. Zainoulline,
\textit{Invariants, torsion indices and oriented cohomology of complete
  flags,} Ann. Sci. \'Ecole Norm. Sup. (4) 46 (2013), no. 3, 405–448. \MR{3099981}


\bibitem[CZZ12]{CZZ1}
B. Calm\`es, K. Zainoulline,  and C. Zhong, 
\textit{A coproduct structure on the formal affine Demazure algebra}, to appear in Math. Zeitschrift, \arxiv{1209.1676}

\bibitem[CZZ13]{CZZ2}
B. Calm\`es, K. Zainoulline,  and C. Zhong, 
\textit{Push-pull operators on the formal affine Demazure algebra and its dual}, Preprint, (2013). \arxiv{1312.0019}

\bibitem[CZZ14]{CZZ3}
B. Calm\`es, K. Zainoulline,  and C. Zhong, \textit{Equivariant oriented cohomology of flag varieties}, preprint, (2014). \arxiv{1409.7111}

\bibitem[CZZ]{CZZ} B. Calm\`es, K. Zainoulline,  and C. Zhong, {\em Formal affine Demazure algebras associated to generalized Cartan matrices}, in preparation.

\bibitem[CG97]{CG} N.~Chriss and V.~Ginzburg, {\sl Representation theory and complex
geometry}, Birkh\"auser, Boston-Basel-Berlin, 1997. \MR{2838836}

\bibitem[Dem73]{Dem}
M. Demazure,
\textit{Invariants sym\'etriques entiers des groupes de Weyl et torsion,}
{Invent. Math.} {\bf 21} (1973), 287–301. \MR{0342522}

\bibitem[Des09]{Des} D. Deshpande, 
{\em Algebraic cobordism of classifying spaces}, preprint, (2009).
\arxiv{0907.4437v1}

\bibitem[EG98]{EG} D. Edidin and W. Graham, 
{\em Equivariant intersection theory}, Invent. Math., 131(3),595-634, (1998). \MR{1614555}

\bibitem[Fed94]{Fed} G. Felder, {\it Elliptic quantum groups}. XIth International Congress of Mathematical Physics (Paris, 1994), 211–218, Int. Press, Cambridge, MA, 1995. \MR{1370676}

\bibitem[Fr68]{Fr} A. Fr$\ddot{o}$hlich, 
{\em Formal groups}, Lecture Notes in Mathematics, No. 74, Springer, Berlin (1968). \MR{0242837}


\bibitem[Ful98]{Ful} W.~Fulton, {\em Intersection theory}, Springer-Verlag, Berlin, 1998. \MR{1644323}

\bibitem[Gep06]{Gep} D.~Gepner, {\em Equivariant elliptic cohomology and homotopy topoi}, Ph.D. thesis, University of Illinois, (2006).

\bibitem[Gin85]{Gin85} V. Ginzburg, {\em Deligne-Langlands conjecture and representations of affine Hecke algebras}, Preprint, Moscow (1985).

\bibitem[Gin98]{Gin} V. Ginzburg, 
{\it Geometric methods in the representation theory of Hecke algebras and quantum groups}, preprint, (1998). \arxiv{9802004}.

\bibitem[GKV95]{GKV95} V.~Ginzburg, M.~Kapranov, and E.~Vasserot, {\em Elliptic algebras and equivariant elliptic cohomology},
Preprint, (1995). \arxiv{9505012}


\bibitem[GKV97]{GKV97}
V.~Ginzburg, M.~ Kapranov, and E.~Vasserot,
{\em Residue construction of Hecke algebras}, Adv. Math. {\bf 128} (1997), no. 1, 1--19. \MR{1451416}

\bibitem[Gr94a]{Gr94} I.~Grojnowski, {\em Delocalized equivariant elliptic cohomology}, Elliptic cohomology, London Math. Soc. Lecture Note Ser., {\bf 342}, Cambridge Univ. Press, (2007), 111--113. \MR{2330509}

\bibitem[Gr94b]{Gr} I.~Grojnowski, {\em Delocalized equivariant elliptic cohomology}, Elliptic cohomology, London Math. Soc. Lecture Note Ser., {\bf 342}, Cambridge Univ. Press, (2007), 111--113. \MR{2330509}


\bibitem[GTL10]{GTL} S.~Gautam and  V.~Toledano Laredo, {\em Yangians and quantum loop algebras.} Selecta Mathematica, to appear. Preprint available at \arxiv{1012.3687}.


\bibitem[HM13]{HM} J. Heller and J. Malag\'on-L\'opez, 
{\em Equivariant algebraic cobordism}, J. Reine Angew. Math., 684, 87-112, (2013). \MR{3181557 }



\bibitem[Hop02]{Hop_ICM} M.~Hopkins, {\em Algebraic topology and modular forms}, Proceedings of the International Congress of Mathematicians, Vol. I (Beijing, 2002), 291-317, Higher Ed. Press, Beijing, 2002. \MR{1989190}

\bibitem[HMSZ12]{HMSZ}
A. Hoffnung, J. Malag\'{o}n-L\'{o}pez, A. Savage, and K. Zainoulline, \textit{Formal Hecke algebras and algebraic oriented cohomology
theories,} to appear in Selecta Math. \arxiv{1208.4114}

\bibitem[KK86]{KK86}
B. Kostant and S. Kumar,
 \textit{The nil {H}ecke ring and cohomology of {$G/P$} for a {K}ac-{M}oody group {$G^*$},}
Adv. in Math. {\bf 62} (1986), no. 3, 187--237. \MR{0866159}

\bibitem[KK90]{KK90}
B. Kostant and S. Kumar,
\textit{$T$-equivariant $K$-theory of generalized flag varieties,}
{J. Differential geometry} 32 (1990), 549--603. \MR{1072919}

\bibitem[KL87]{KL} D. Kazhdan and G. Lusztig, {\em Proof of the Deligne-Langlands conjecture for Hecke algebras}, Invent. Math. 87 (1987), no. 1, 153--215.  \MR{0862716}


\bibitem[Kr12]{Kr} A. Krishna, 
{\em Equivariant cobordism of schemes}, Doc. Math., 17, 95-134, (2012). \MR{2889745}

\bibitem[LM07]{LM} M.~Levine and F.~Morel,
{\em Algebraic cobordism theory}, Springer, Berlin, 2007. \MR{2286826}

\bibitem[Lur09]{Lur} J.~Lurie, {\em A survey of elliptic cohomology},  Algebraic topology, 219--277, Abel Symp., 4, Springer, Berlin, 2009. \MR{2597740}


\bibitem[Lus85]{Lus85} G.~Lusztig, {\em Equivariant $K$-theory and representations of Hecke algebras}. Proc. Amer. Math. Soc. {\bf 94} (1985), no. 2, 337-342. \MR{0784189}


\bibitem[Lus88]{Lus88} 
G.~Lusztig, {\em Cuspidal local systems and graded Hecke algebras. I}, Inst. Hautes \'Etudes Sci. Publ. Math. 67 (1988), 145-202. \MR{0972345}

\bibitem[MV99]{MV} F.~Morel and V.~Voevodsky, {$\mathbb{A}^1$-homotopy theory of schemes},
Publ. Math. IHES \textbf{90} (1999), 45--143. \MR{1813224}

\bibitem[Nak01]{Nak99} H.~Nakajima, {\em Quiver varieties and finite dimensional
representations of quantum affine algebras},  J. Amer. Math. Soc. \textbf{14} (2001), no. 1, 145每238. \MR{1808477} \arxiv{9912158}


\bibitem[PPR07]{PPR1}I.~Panin, K. Pimenov and O.~R\"{o}ndigs,  {\em On Voevodsky's algebraic
$K$-theory spectrum $BGL$}, Preprint, 48 pages (2007). \arxiv{0709.3905v1}.


\bibitem[PS04]{PaninRR}
I.~Panin and A.~Smirnov, {\em Riemann-Roch theorems for oriented cohomology}. Axiomatic, enriched and motivic homotopy theory, 261--333, NATO Sci. Ser. II Math. Phys. Chem., 131, Kluwer Acad. Publ., Dordrecht, 2004. \MR{2061857}

\bibitem[PS09]{PaninSmirnovII}
I.~Panin, {\em
Oriented cohomology theories of algebraic varieties. II (After I. Panin and A. Smirnov)}.
Homology, Homotopy Appl. \textbf{11} (2009), no. 1, pp. 349--405. \MR{2529164}

\bibitem[Th87]{Th} R. W. Thomason, 
{\em Algebraic K-theory of group scheme actions}, in Algebraic topology and algebraic K-theory (Princeton, N.J., 1983), Ann. of Math. Stud. 113, 539-563, (1987). \MR{0921490}

\bibitem[To99]{To} B. Totaro, 
{\em The Chow ring of a classifying space}, Algebraic K-theory (Seattle, WA, 1997), Proc. Sympos. Pure Math., 67, 249-281, (1999). \MR{1743244}

\bibitem[Va00]{Va00} M.~Varagnolo,
{\em Quiver Varieties and Yangians}, Lett. Math. Phys.~\textbf{53} (2000), no.~4, 273--283.
\MR{1818101}


\bibitem[Vish07]{Vish} A. Vishik, {\em Symmetric operations in algebraic cobordisms}, Adv. Math. 213 (2007), no. 2, 489-552.
\MR{2332601}


\bibitem[YZ14]{YZ} Y.~Yang and G.~Zhao, {\em Formal cohomological Hall algebra of a quiver}, Preprint, 2014. \arxiv{1407.7994}

\bibitem[Zha]{Zha} G.~Zhao, {\em Quiver varieties and elliptic quantum groups}, in preparation.

\bibitem[Zho13]{Zh13} C. Zhong, {\em On the formal affine Hecke algebra}, to appear in J. Inst. Math. Jussieu, \arxiv{1301.7497}

 
\bibitem[ZZ]{ZZ} G.~Zhao and C.~Zhong, {\em The elliptic Demazure-Lusztig operator and representations of the elliptic affine Hecke algebra}, in preparation. 


\end{thebibliography}
\end{document}